\documentclass[a4paper,12pt]{amsart}
\usepackage{amsmath,amssymb,amsfonts,amsthm}

\usepackage[alphabetic]{amsrefs}
\usepackage{bm}

\usepackage{graphicx}
\usepackage{ascmac}
\usepackage[all]{xy}
\usepackage{amsthm}
\usepackage[top=25mm,bottom=25mm,left=25mm,right=25mm]{geometry}
\usepackage{tikz}
\usetikzlibrary{cd,arrows,matrix,backgrounds,shapes,positioning,calc,decorations.markings,decorations.pathmorphing,decorations.pathreplacing}
\tikzset{blackv/.style={circle,fill=black,inner sep=3pt,outer sep=3pt},
         whitev/.style={circle,fill=white,draw=black,inner sep=3pt,outer sep=3pt},
         blabel/.style={circle,draw=black,inner sep=1.5pt,outer sep=0pt},
         redv/.style={circle,fill=red,inner sep=3pt,outer sep=3pt},
         block/.style={draw,rectangle split,rectangle split horizontal,rectangle split parts=#1},
         symbol/.style={
           draw=none,
           every to/.append style={
             edge node={node [sloped, allow upside down, auto=false]{$#1$}}}}
}
\usepackage{multirow}
\usetikzlibrary{intersections,arrows,calc,decorations.markings}
\newtheorem{theorem}{Theorem}[section]
\usepackage{mathtools,leftidx}

\newtheorem{proposition}[theorem]{Proposition}

\newtheorem{corollary}[theorem]{Corollary}
\newtheorem{lemma}[theorem]{Lemma}

\theoremstyle{definition}
\newtheorem{remark}[theorem]{Remark}
\newtheorem{example}[theorem]{Example}
\newtheorem{definition}[theorem]{Definition}

\newtheorem{question}[theorem]{Question}

\makeatletter
 
 \@addtoreset{equation}{section}
\makeatother
\def\bb{\mathbf{b}}
\def\cc{\mathbf{c}}

\def\ee{\mathbf{e}}

\def\gg{\mathbf{g}}

\def\xx{\mathbf{x}}
\def\yy{\mathbf{y}}

\def\TT{\mathbb{T}}

\def\PP{\mathbb{P}}
\def\ZZ{\mathbb{Z}}

\def\Acal{\mathcal{A}}

\def\Fcal{\mathcal{F}}
\def\Xcal{\mathcal{X}}

\def\mod{\opname{mod}\nolimits}

\def\Hasse{\opname{Hasse}\nolimits}
\newcommand{\opname}[1]{\operatorname{\mathsf{#1}}}

\newcommand{\Hom}{\opname{Hom}}
\newcommand{\End}{\opname{End}}
\newcommand{\Ext}{\opname{Ext}}
\newcommand{\Fac}{\opname{Fac}}
\newcommand{\add}{\opname{add}\nolimits}
\newcommand{\ind}{\opname{ind}\nolimits}

\newcommand{\sst}[1]{\substack{#1}}

\newcommand{\myboxa}[1]{\framebox[8em]{\rule{0pt}{4.5ex}#1}}
\newcommand{\myboxb}[1]{\framebox[9.5em]{\rule{0pt}{4.5ex}#1}}
\newcommand{\myboxc}[1]{\framebox[9.5em]{\rule{0pt}{4.5ex}#1}}

\DeclareMathOperator{\stautilt}{\mathsf{s}\tau\textrm{-}\mathsf{tilt}}
\def\stautiltminus{\mathsf{s}\tau^{-}\textrm{-}\mathsf{tilt}}
\def\apPhi{\Phi_{\geq -1}}
\title[{Lattice structure in cluster algebra of finite type}]{Lattice structure in cluster algebra of finite type and non-simply-laced Ingalls-Thomas bijection}
\author{Yasuaki Gyoda}
\address{Graduate School of Mathematical Sciences, The University of Tokyo, 3-8-1 Komaba Meguro-ku Tokyo 153-8914, Japan}
\email{gyoda-yasuaki@g.ecc.u-tokyo.ac.jp}
\begin{document}
\begin{abstract}
In this paper, we demonstrate that the lattice structure of a set of clusters in a cluster algebra of finite type is anti-isomorphic to the torsion lattice of a certain Gei\ss-Leclerc-Schr\"oer (GLS) path algebra and to the $c$-Cambrian lattice. We prove this by explicitly describing the exchange quivers of cluster algebras of finite type. Specifically, we prove that these quivers are anti-isomorphic to those formed by support $\tau$-tilting modules in GLS path algebras and to those formed by $c$-clusters consisting of almost positive roots.
\end{abstract}
\maketitle
\tableofcontents
\section{Introduction}
\subsection{Backgrounds}
A \emph{cluster algebra} is a commutative algebra introduced by Fomin and Zelevinsky in \cite{fzi} and is generated by special elements called \emph{cluster variables}. Several cluster variables form a set referred to as a \emph{cluster}, and all cluster variables can be generated inductively by repeating \emph{mutation} operations, which remove one variable from a cluster and add another variable to it.

Since the discovery of cluster algebras, a flurry of research has been conducted on their combinatorial structure. Consequently, it has been discovered that this structure also appears in Lie theory, representation theory of algebras, Teichm\"uller theory, and other areas. In particular, cluster algebras of finite type are strongly related to Dynkin diagrams due to their classification method \cite{fzii}, and they have a good relation with root systems of finite type and path algebras of finite representation type, which are also classified using Dynkin diagrams. The relationship between the three has been widely researched.

The famous Gabriel's theorem, which examines the relation between path algebras of finite representation type and Dynkin root systems, was first stated by Gabriel in the 1970s \cite{gab}. This theorem establishes a bijection between the isomorphism classes of indecomposable modules of path algebras of finite representation type and the positive roots in the corresponding simply-laced root system.

According to \cite{fzii}, there exists a canonical bijection between cluster variables and almost positive roots (elements of the union set of positive roots and negative simple roots) for a relation between cluster algebras and root systems. It also induces the bijection between clusters and maximal compatible sets of almost positive roots.

Around the same time, \cite{bmrrt} defined the cluster category in the representation theory of algebras. The cluster category is a triangulated category that includes the finitely generated module category of a path algebra. Isomorphism classes of indecomposable objects of a cluster category correspond bijectively with the almost positive roots of the corresponding simply-laced root system (or cluster variables of the corresponding cluster algebra). In \cite{bmrrt}, cluster-tilting objects that have a bijective correspondence with clusters in the corresponding cluster algebra are also introduced. The purpose of the cluster category is to extend the module category in order to provide a complete correspondence with cluster variables and almost positive roots. On the other hand, pairs of objects in the module category that have a bijective correspondence with cluster variables and with almost positive roots have also been discovered. They are introduced as $\tau$-tilting pairs in \cites{ing-tho,air}.

These studies have demonstrated that cluster algebras, root systems, and path algebras are strongly linked through a common combinatorial structure.

\subsection{Main results: Isomorphisms between three lattice structures}

The \emph{exchange quiver} of a cluster algebra is an oriented graph with non-labeled clusters as vertices and edges connecting two clusters that are switched by a mutation. The orientation of edges are determined by the sign of the vectors associated with each cluster, called the \emph{$c$-vectors}. The longest path of the exchange quiver is known as the \emph{maximal green sequence}, and it is used to calculate the refined Donaldson-Thomas invariant \cite{kont-soib} when a cluster algebra is of skew-symmetric type (for details, refer to \cite{kel-dem}). This major finding of this study is that a set of clusters has a lattice structure resulting from an exchange quiver.

Let $\Phi$ represent a root system and $c$ a Coxeter element of $\Phi$. The primary contribution of this study is the following theorem:

\begin{theorem}\label{lattice-main}\noindent
\begin{itemize}\setlength{\leftskip}{-0.7cm}
\item [(1)]The set of clusters in a cluster algebra of finite type $\Acal(\Phi,c)$\footnote{The symbol $\Acal(\Phi,c)$ refers to $\Acal(B^c;t_0)$ in Section 2 or later.} has a lattice structure whose Hasse quiver is the exchange quiver of $\Acal(\Phi,c)$; 
\item[(2)]it is anti-isomorphic to the torsion lattice of the GLS path algebra $H(C_\Phi, D_\Phi, c)$;
\item[(3)]it is anti-isomorphic to the $c$-Cambrian lattice.
\end{itemize}
\end{theorem}

Theorem \ref{lattice-main} (1) and (3) have already been proved implicitly by \cite{resp2}*{Lemma 5.11}. Here, we prove them in a different way, not relying on framework theory. Here, we will give this relation by a different way.  

In Theorem \ref{lattice-main} (2) and (3), we provide explicit descriptions of the lattice structure introduced into the set of clusters. The \emph{torsion lattice} is composed of torsion classes of an algebra, and the algebra $H(C_\Phi, D_\Phi, c)$ is the \emph{Gei\ss-Leclerc-Schr\"oer (GLS) path algebra} associated with $\Phi$ and $c$, which generalizes the path algebra described in \cite{gls-i}. The \emph{$c$-Combrian lattice} is the lattice determined by the Coxeter element $c$ of the root system $\Phi$ and composed of certain elements in the corresponding Weyl group. This is explained by \cite{reading}.

We will state the key theorem in order to demonstrate Theorem \ref{lattice-main}. Two quivers are introduced: the \emph{support $\tau$-tilting quiver} and the \emph{quiver of $c$-clusters}. The former is a quiver defined by a root system, while the latter is defined by an algebra.

In the representation theory of algebras, for an algebra $A$, special pairs of modules referred to as \emph{basic $\tau$-tilting pairs} of $A$ have a mechanism similar to the mutation of clusters. To be precise, a new basic $\tau$-tilting pair can be obtained by switching indecomposable direct summands. A \emph{support $\tau$-tilting quiver} was defined in the same way as the exchange quiver, in \cite{air}, where the vertex set consists of basic $\tau$-tilting pairs (or their support $\tau$-tilting components). When $A$ has finitely numerous isomorphism classes of basic $\tau$-tilting pairs, the support $\tau$-tilting quiver is quiver-isomorphic to the Hasse quiver of the torsion lattice (see Theorem \ref{supp-tau-torsion}).

For a root system $\Phi$ and a Coxeter element $c$, there are sets of roots called \emph{$c$-clusters}, which possess the same mutation structure as clusters in cluster algebras, and there is a ``$c$-cluster version of the exchange quiver" called the \emph{quiver of $c$-clusters}, whose vertex set consists of $c$-clusters. This quiver was defined in \cite{re-sp}. It is quiver-isomorphic to the Hasse quiver of the $c$-Cambrian lattice (see Theorem \ref{cambrian-cluster}).

The following theorem implies the coincidence of these quivers:
\begin{theorem}[Theorems \ref{lattice-op-iso-psi}, \ref{lattice-op-iso}, \ref{lattice-iso-phi}]\label{main-intro}
We fix a root system $\Phi$ and a Coxeter element $c$, and we denote their corresponding cluster algebra and GLS path algebra to $\Phi$ and $c$ by $\Acal(\Phi,c)$ and $H=H(C_\Phi,D_{\Phi},c)$, respectively.
The maps $\tilde{\phi}_c,\tilde{\psi}_c,\tilde{\theta}_c$ are quiver isomorphisms that make the following diagram commutative:
  
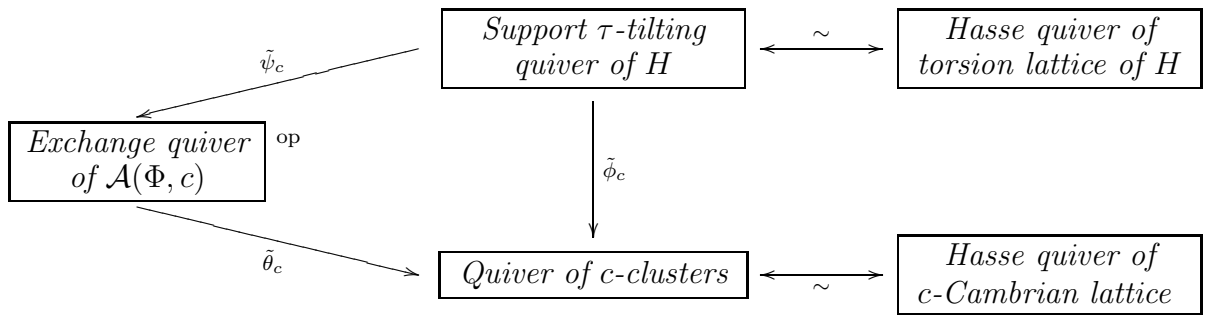
\begin{align*}
   \begin{xy}
   (-20,0)*{\text{\myboxa{}}},(40,15)*{\text{\myboxb{}}},(100,15)*{\text{\myboxc{}}},(100,-15)*{\text{\myboxc{}}},
   (-20,2.3)*{\text{Exchange quiver}}="A1",(-20,-2.3)*{\text{of $\Acal(\Phi,c)$}}="A2",(0,2.3)*{^{\mathrm{op}}},(-20,6)*{}="A'",(-20,-6)*{}="A''",(17,15)*{}="B'",(17,-15)*{}="D'",(40,8)*{}="B''",(40,-10)*{}="D''",(40,-10)*{}="D''",(62,15)*{}="E",(78,15)*{}="F",(62,-15)*{}="G",(78,-15)*{}="H",(40,17.3)*{\text{Support $\tau$-tilting}}, (40,12.7)*{\text{quiver of $H$}}="B",(40,-15)*{\text{\fbox{ Quiver of $c$-clusters }}}="D",(100,17.3)*{\text{Hasse quiver of}}, (100,12.7)*{\text{torsion lattice of $H$}},(100,-12.7)*{\text{Hasse quiver of}},(100,-17.3)*{\text{$c$-Cambrian lattice }},\ar^{\tilde\psi_c}@{<-} "A'";"B'"\ar_{\tilde\theta_c}@{->} "A''";"D'"\ar^{\tilde\phi_c}@{->}"B''";"D''"\ar^{\sim}@{<->}"E";"F"\ar_{\sim}@{<->}"G";"H"
\end{xy}
\end{align*}
where $-^{\mathrm{op}}$ means the opposite quiver.
\end{theorem}

Theorem \ref{main-intro} immediately follows Theorem \ref{lattice-main}. Moreover, we have the following corollary:

\begin{corollary}[Corollary \ref{main}]\label{main2-intro}
For any root system $\Phi$ and any Coxeter element $c$, the torsion lattice of the GLS path algebra $H(C_\Phi,D_{\Phi},c)$ is isomorphic to the $c$-Cambrian lattice.
\end{corollary}

\begin{remark}
This corollary is not a new result. We can prove it by combining previous studies. See Remark \ref{important-remark}.      
\end{remark}

Let us examine the relation between Corollary 1.3 and the following famous result by Ingalls and Thomas (see also \cite{dirrt}*{Theorem 7.1}):

\begin{theorem}[\cite{ing-tho}*{Theorem 4.3}]\label{ingalls-thomas}
For any Dynkin quiver $Q$ and the corresponding Coxeter element $c$, the torsion lattice of the path algebra $KQ$ is isomorphic to the $c$-Cambrian lattice.
\end{theorem}

The quiver $Q$ contains information concerning a simply-laced Dynkin diagram and its orientation; therefore, it is equivalent to two pieces of information on the simply-laced root system $\Phi$ and the Coxeter element $c$. Moreover, the ordinary path algebra $KQ$ is isomorphic to the GLS algebra $H(C_\Phi, I_n, c)$, where $I_n$ is the identity matrix. Therefore, Theorem \ref{ingalls-thomas} is the specialization of Corollary \ref{main2-intro}.

Back to the topic of Theorem \ref{main-intro}. The key to the proof of Theorem \ref{main-intro} is based a comparison between the $\tau$-tilting quiver of GLS path algebra $H$ and its opposite algebra $H^\mathrm{op}$. We provide an explicit relation between these two quivers by the following proposition:
\begin{proposition}[(1) Lemma \ref{DM-tau-tilting}, (2) Propostion \ref{graph-iso-dual}, (3) Theorem \ref{condition-tilting-true}]\label{comparing-dual}
\noindent
\begin{itemize}\setlength{\leftskip}{-0.7cm}
 \item[(1)] The standard duality $D\colon\mod H\to\mod H^{\mathrm{op}}$ provides a bijection between the support $\tau$-tilting modules of $H$ and $H^\mathrm{op}$;
 \item[(2)] it induces a graph isomorphism between the support $\tau$-tilting quivers of $H$ and $H^{\mathrm{op}}$;
 \item[(3)] when there exists an arrow $M\to M'$ in the support $\tau$-tilting quiver of $H$,
 \begin{itemize}\setlength{\leftskip}{-0.7cm}
    \item [(i)] $|M|\neq|M'|$ holds if and only if $H^{\mathrm{op}}$ has an arrow $DM\to DM'$,
    \item [(ii)] $|M|=|M'|$ holds if and only if $H^{\mathrm{op}}$ has an arrow $DM'\to DM$,
\end{itemize}
where $|M|$ is the number of isomorphism classes of indecomposable summands of $M$.
\end{itemize}
\end{proposition}
For example, the following two quivers are support $\tau$-tilting quivers of $H$ (left one) and $H^{\mathrm{op}}$ (right one) for $H=K(1\leftarrow 2\leftarrow3)$, and they are arranged so that their corresponding vertices are at the same position by standard duality $D$.
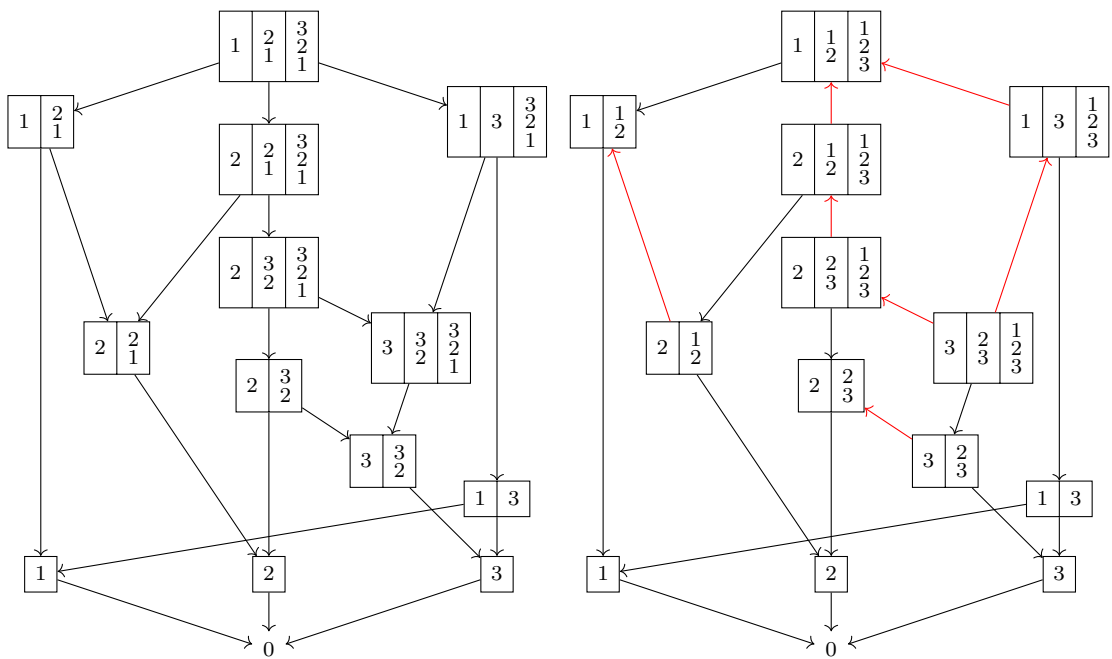
 \[\begin{tikzpicture}
      \node[block=3] (121321) at (0,10) {$\sst{1}$\nodepart{two}$\sst{2 \\ 1}$\nodepart{three}$\sst{3 \\ 2 \\ 1}$};
      \node[block=2] (121) at (-3,9) {\nodepart{one}$\sst{1}$\nodepart{two}$\sst{2 \\ 1}$};
      \node[block=3] (221321) at (0,8.5) {\nodepart{one}$\sst{2}$\nodepart{two}$\sst{2 \\ 1}$\nodepart{three}$\sst{3 \\ 2 \\ 1}$};
      \node[block=3] (13321) at (3,9) {\nodepart{one}$\sst{1}$\nodepart{two}$\sst{3}$\nodepart{three}$\sst{3 \\ 2 \\ 1}$};
      \node[block=2] (221) at (-2,6) {\nodepart{one}$\sst{2}$\nodepart{two}$\sst{2 \\ 1}$};
      \node[block=3] (232321) at (0,7) {\nodepart{one}$\sst{2}$\nodepart{two}$\sst{3 \\ 2}$\nodepart{three}$\sst{3 \\ 2 \\ 1}$};
      \node[block=2] (232) at (0,5.5) {\nodepart{one}$\sst{2}$\nodepart{two}$\sst{3 \\ 2}$};
      \node[block=3] (332321) at (2,6) {\nodepart{one}$\sst{3}$\nodepart{two}$\sst{3 \\ 2}$\nodepart{three}$\sst{3 \\ 2 \\ 1}$};
      \node[block=2] (13) at (3,4) {\nodepart{one}$\sst{1}$\nodepart{two}$\sst{3}$};
      \node[rectangle,draw] (2) at (0,3) {\nodepart{one}$\sst{2}$};
      \node[block=2] (332) at (1.5,4.5) {\nodepart{one}$\sst{3}$\nodepart{two}$\sst{3 \\ 2}$};
      \node[rectangle,draw] (3) at (3,3) {\nodepart{one}$\sst{3}$};
      \node[rectangle,draw] (1) at (-3,3) {\nodepart{one}$\sst{1}$};
      \node (0) at (0,2) {$\sst{0}$};

      \draw[->] (121321)--(121);
      \draw[->] (121321)--(221321);
      \draw[->] (121321)--(13321);
      \draw[->] (121)--(221);
      \draw[->] (121)--(1);
      \draw[->] (221321)--(221);
      \draw[->] (221321)--(232321);
      \draw[->] (13321)--(332321);
      \draw[->] (13321)--(13);
      \draw[->] (221)--(2);
      \draw[->] (232321)--(232);
      \draw[->] (232321)--(332321);
      \draw[->] (232)--(332);
      \draw[->] (232)--(2);
      \draw[->] (332321)--(332);
      \draw[->] (13)--(1);
      \draw[->] (13)--(3);
      \draw[->] (332)--(3);
      \draw[->] (1)--(0);
      \draw[->] (2)--(0);
      \draw[->] (3)--(0);
    \end{tikzpicture}
    \hspace{3mm}
    \begin{tikzpicture}
      \node[block=3] (121321) at (0,10) {$\sst{1}$\nodepart{two}$\sst{1 \\ 2}$\nodepart{three}$\sst{1 \\ 2 \\ 3}$};
      \node[block=2] (121) at (-3,9) {\nodepart{one}$\sst{1}$\nodepart{two}$\sst{1 \\ 2}$};
      \node[block=3] (221321) at (0,8.5) {\nodepart{one}$\sst{2}$\nodepart{two}$\sst{1 \\ 2}$\nodepart{three}$\sst{1 \\ 2 \\ 3}$};
      \node[block=3] (13321) at (3,9) {\nodepart{one}$\sst{1}$\nodepart{two}$\sst{3}$\nodepart{three}$\sst{1 \\ 2 \\ 3}$};
      \node[block=2] (221) at (-2,6) {\nodepart{one}$\sst{2}$\nodepart{two}$\sst{1 \\ 2}$};
      \node[block=3] (232321) at (0,7) {\nodepart{one}$\sst{2}$\nodepart{two}$\sst{2 \\ 3}$\nodepart{three}$\sst{1 \\ 2 \\ 3}$};
      \node[block=2] (232) at (0,5.5) {\nodepart{one}$\sst{2}$\nodepart{two}$\sst{2 \\ 3}$};
      \node[block=3] (332321) at (2,6) {\nodepart{one}$\sst{3}$\nodepart{two}$\sst{2 \\ 3}$\nodepart{three}$\sst{1 \\ 2 \\ 3}$};
      \node[block=2] (13) at (3,4) {\nodepart{one}$\sst{1}$\nodepart{two}$\sst{3}$};
      \node[rectangle,draw] (2) at (0,3) {\nodepart{one}$\sst{2}$};
      \node[block=2] (332) at (1.5,4.5) {\nodepart{one}$\sst{3}$\nodepart{two}$\sst{2 \\ 3}$};
      \node[rectangle,draw] (3) at (3,3) {\nodepart{one}$\sst{3}$};
      \node[rectangle,draw] (1) at (-3,3) {\nodepart{one}$\sst{1}$};
      \node (0) at (0,2) {$\sst{0}$};

      \draw[->] (121321)--(121);
      \draw[<-,red] (121321)--(221321);
      \draw[<-,red] (121321)--(13321);
      \draw[<-,red] (121)--(221);
      \draw[->] (121)--(1);
      \draw[->] (221321)--(221);
      \draw[<-,red] (221321)--(232321);
      \draw[<-,red] (13321)--(332321);
      \draw[->] (13321)--(13);
      \draw[->] (221)--(2);
      \draw[->] (232321)--(232);
      \draw[<-,red] (232321)--(332321);
      \draw[<-,red] (232)--(332);
      \draw[->] (232)--(2);
      \draw[->] (332321)--(332);
      \draw[->] (13)--(1);
      \draw[->] (13)--(3);
      \draw[->] (332)--(3);
      \draw[->] (1)--(0);
      \draw[->] (2)--(0);
      \draw[->] (3)--(0);
    \end{tikzpicture}
    \]
In the above diagrams, we abbreviate $\sst{1}\oplus\sst{2 \\ 1}\oplus \sst{3\\ 2 \\ 1}$ to $\begin{tikzpicture}[baseline=-2mm]
      \node[block=3] (121321) at (0,0) {$\sst{1}$\nodepart{two}$\sst{2 \\ 1}$\nodepart{three}$\sst{3 \\ 2 \\ 1}$};\end{tikzpicture}$ (The small number $i$ in the box indicates that a one-dimensional vector space is placed on the quiver's vertex $i$).
The arrows connecting the modules are shown in black when two modules at either ends satisfy $|M|\neq |M'|$, and the arrows are shown in red when $|M|=|M'|$. According to theorem \ref{comparing-dual}, arrows in two quivers that are in the same position have the same orientation if they are black and opposite orientation if they are red. It is proven that $\tilde{\theta}_c$ is a quiver isomorphism using Proposition \ref{comparing-dual}.

\subsection*{Organization}In Section 2, we recall the definitions and properties of $c$-clusters, cluster algebras, $\tau$-tilting theory, GLS path algebras, the mathematical objects discussed in this paper. In Section 3, we describe the graph isomorphisms $\tilde{\phi}_c,\tilde{\psi}_c,\tilde{\theta}_c$ in the diagram above. In Section 4, we prove that $\tilde{\psi}_c$ is a quiver isomorphism. In Section 5, we compare quivers of one algebra and its opposite. This is the key to the subsequent section. In Section 6, we prove that $\tilde{\phi}_c$ is a quiver isomorphism. In Section 7, we prove that $\tilde{\phi}_c$ is a quiver isomorphism. This can be proven by using other two isomorphisms immediately, but we will give a direct proof that does not rely on these isomorphisms. In Section 8, after recalling the definitions and properties of the torsion lattice and $c$-Combrian lattice, we present Corollary \ref{main2-intro}.
\subsection*{Acknowledgments}
The author thanks Haruhisa Enomoto for raising the important issues that inspired him to write this paper. The author thanks Yoshiyuki Kimura for introducing him to the important previous studies. The author thanks Arashi Sakai and Haruhisa Enomoto for helpful advice on Section 7. The author thanks Osamu Iyama for helpful comment about the oganization of the paper. The author thanks Nathan Reading for helpful comment about previous studies. The author thanks Kota Murakami for useful comment for the GLS path algebra. The author thanks Riku Fushimi for careful reading and useful comments. This work was supported by JSPS KAKENHI Grant Number JP22J00523. 

\section{Preliminaries}
\subsection{Cluster algebras}
We recall that a \emph{semifield} $\PP$ is an abelian multiplicative group with an distributive addition $\oplus$ over the multiplication.

Let Trop$(u_1,\dots, u_\ell)$ be the free (multiplicative) abelian group generated by $u_1,\dots,u_\ell$. Then, $\text{Trop}(u_1,\dots,u_{\ell})$ is a semifield with the addition defined by
\begin{align}
\prod_{j=1}^\ell u_j^{a_j} \oplus \prod_{j=1}^{\ell} u_j^{b_j}=\prod_{j=1}^{\ell} u_j^{\min(a_j,b_j)}.\nonumber
\end{align}
We refer to it as the \emph{tropical semifield} of $u_1,\dots,u_\ell$.

Let $B=(b_{ij})$ be an $n \times n$ integer matrix. If there exists a positive integer diagonal matrix $S=\mathrm{diag}(s_1,\dots,s_n)$ such that $SB$ is skew-symmetric, then we say that $B$ is {\em skew-symmetrizable}. The matrix $S$ is called a \emph{skew-symmetrizer} of $B$.

We fix a positive integer $n$ and a semifield $\PP$. Let $\mathbb{Z}\PP$ be the group ring of the multiplicative group $\PP$. Then $\mathbb{Z}\PP$ is a domain (\cite{fzi}*{Section 5}). Thus its total quotient ring is a field $\mathbb{Q}\PP$. Let $\mathcal{F}$ be the field of rational functions in $n$ indeterminates with coefficients in $\mathbb{Q}\PP$.

A \emph{(labeled) seed} of rank $n$ over $\PP$ is a triplet $(\mathbf{x}, \mathbf{y}, B)$, where
\begin{itemize}
\item $\mathbf{x}=(x_1, \dots, x_n)$ represents an $n$-tuple of elements of $\mathcal F$ forming a free generating set of $\mathcal F$. $\mathbf{x}$ is called a {\em cluster}, and the elements in $\mathbf{ x}$ are referred to as \emph{cluster variables}.
\item $\mathbf{y}=(y_1, \dots, y_n)$ represents an $n$-tuple of elements of $\PP$. The elements in $\mathbf{ y}$ are called \emph{coefficients}.
\item $B=(b_{ij})$ is an $n \times n$ skew-symmetrizable integer matrix that is referred to as the {\em exchange matrix} of $(\xx,\yy,B)$.
\end{itemize}

For an integer $b$, we denote $[b]_+=\max(b,0)$. Clearly, we have $b=[b]_+-[-b]_+$.

\begin{definition}[Matrix mutation] Let $A=(a_{ij})$ be an $m\times n$ integer matrix with $m\geq n>0$ and $k\in\{1,\cdots,n\}$. The {\em mutation} $\mu_k(A)$ of $A$ in direction $k$ is the new matrix $A^\prime=(a_{ij}^\prime)$ given by
\begin{eqnarray}
\label{eq:matrix mutation}
a'_{ij}=\begin{cases}-a_{ij} ,&\text{if $i=k$ or $j=k$;} \\
a_{ij}+\left[ a_{ik}\right] _{+}a_{kj}+a_{ik}\left[ -a_{kj}\right]_+ ,&\text{otherwise.}\nonumber
\end{cases}
\end{eqnarray}
\end{definition}
It can be verified that if the upper $n\times n$ submatrix of $A_{m\times n}$ is skew-symmetrizable, then $$\mu_k(\mu_k(A))=A.$$
We fix $k\in\{1,\dots,n\}$. If $a_{ki}$ is negative for each $i\in\{1,\dots,n\}$ then $\mu_k(A)$ is called the \emph{sink mutation}, and if it is positive, then $\mu_k(A)$ is called the \emph{source mutation}.
Let $(\mathbf{x}, \mathbf{y}, B)$ be a seed of rank $n$ over $\PP$ and $k \in\{1,\dots, n\}$. The \emph{seed mutation} $\mu_k$ in direction $k$ transforms $(\mathbf{x}, \mathbf{y}, B)$ to a new seed $\mu_k(\mathbf{x}, \mathbf{y}, B)=(\mathbf{x'}, \mathbf{y'}, B')$ defined as follows:
\begin{itemize}
\item $B^\prime$ is the matrix mutation of $B$ in direction $k$, that is, $B^\prime=\mu_k(B)$;
\item the coefficients in $\mathbf{y'}=(y'_1, \dots, y'_n)$ are given by
\begin{eqnarray}
\label{eq:y-mutation}
y'_j=
\begin{cases}
y_{k}^{-1}, &\text{if $j=k$;} \\
y_j y_k^{[b_{kj}]_+}(y_k \oplus 1)^{-b_{kj}} ,&\text{otherwise.}
\end{cases}\nonumber
\end{eqnarray}
\item The cluster variables in $\mathbf{x'}=(x'_1, \dots, x'_n)$ are given by
\begin{align}\label{eq:x-mutation}
x'_j=\begin{cases}\dfrac{y_k\mathop{\prod}\limits_{i=1}^{n} x_i^{[b_{ik}]_+}+\mathop{\prod}\limits_{i=1}^{n} x_i^{[-b_{ik}]_+}}{(y_k\oplus 1)x_k}, &\text{if $j=k$;}\\
x_j, &\text{otherwise.}
\end{cases}
\end{align}
\end{itemize}

It can be verified that the seed mutation $\mu_k$ is an involution, that is, $\mu_k^2({\bf x},{\bf y}, B)=({\bf x},{\bf y}, B)$.

Let $\mathbb{T}_n$ be the \emph{$n$-regular tree} whose edges are labeled with the numbers $1, \dots, n$ such that the $n$ edges emanating from each vertex have different labels. A graph \begin{xy}(0,1)*+{t}="A",(10,1)*+{t'}="B",\ar@{-}^k"A";"B" \end{xy}
indicates that the two vertices $t,t'\in \mathbb{T}_n$ are joined by an edge labeled by $k$.

\begin{definition}[Cluster pattern]
A \emph{cluster pattern} $\Sigma=\{t\mapsto \Sigma_t\}_{t\in\mathbb T_n}$ of rank $n$ over $\PP$ is an assignment of a seed $\Sigma_t=(\mathbf{x}_t, \mathbf{y}_t,B_t)$ of rank $n$ over $\PP$ to each vertex $t\in \mathbb{T}_n$ such that
for any edge \begin{xy}(0,1)*+{t}="A",(10,1)*+{t'}="B",\ar@{-}^k"A";"B" \end{xy} of $\mathbb T_n$, we have $\Sigma_{t^\prime}=\mu_k(\Sigma_t)$.
\end{definition}

For a seed $\Sigma_t=(\mathbf{x}_t, \mathbf{y}_t,B_t)$, we always denote by:
$$
\mathbf{x}_t=(x_{1;t},\dots,x_{n;t}),\ \mathbf{y}_t=(y_{1;t},\dots,y_{n;t}),\ B_t=(b_{ij;t}).
$$

Sometimes, we select a vertex $t_0$ as the \emph{rooted vertex} of $\mathbb T_n$. The seed at the rooted vertex $t_0$ would be called an \emph{initial seed}. In this case, we denote by:
\begin{align}
\mathbf{x}_{t_0}=(x_1,\dots,x_n),\ \mathbf{y}_{t_0}=(y_1,\dots,y_n),\ B=B_{t_0}=(b_{ij}).\nonumber
\end{align}

\begin{definition}[Cluster algebra]
Let $\Sigma=\{t\mapsto \Sigma_t\}_{t\in\mathbb T_n}$ be a cluster pattern of rank $n$ over $\mathbb P$. The \emph{cluster algebra} $\Acal$ associated with $\Sigma$ is the $\ZZ\PP$-subalgebra of $\Fcal$ given by
 $$\mathcal A=\mathbb{ZP}[x_{1;t},\cdots,x_{n;t}|\;t\in\mathbb T_n].$$
\end{definition}
We always use $\mathcal{A}(B;t_0)$ to denote a cluster algebra with initial exchange matrix $B$ at the rooted vertex $t_0\in\mathbb T_n$.

\begin{example}[Cluster algebra of Type $A_2$] \label{A2} Considering $n=2$, the $2$-regular tree $\TT_2$ is denoted as follows:
\begin{align*}\label{A2tree}
\begin{xy}
(-10,0)*+{\dots}="a",(0,0)*+{t_0}="A",(10,0)*+{t_1}="B",(20,0)*+{t_2}="C", (30,0)*+{t_3}="D",(40,0)*+{t_4}="E",(50,0)*+{t_5}="F", (60,0)*+{\dots}="f"
\ar@{-}^{1}"a";"A"
\ar@{-}^{2}"A";"B"
\ar@{-}^{1}"B";"C"
\ar@{-}^{2}"C";"D"
\ar@{-}^{1}"D";"E"
\ar@{-}^{2}"E";"F"
\ar@{-}^{1}"F";"f"
\end{xy}.
\end{align*}
We set the initial exchange matrix $B=\begin{bmatrix}
 0 & -1 \\
 1 & 0
\end{bmatrix}
$ at the vertex $t_0$. In this case, there are a finite number of coefficients and cluster variables, which are illustrated in Table \ref{A2seed} \cite{fziv}*{Example 2.10}.
\begin{table}[ht]
\begin{equation*}
\begin{array}{|c|cc|cc|}
\hline
&&&&\\[-4mm]
t& \hspace{25mm}\yy_t &&& \xx_t \hspace{30mm}\\
\hline
&&&&\\[-3mm]
0 &y_1 & y_2& x_1& x_2 \\[1mm]
\hline
&&&&\\[-3mm]
1& \dfrac{y_1y_2}{y_2\oplus 1}& \dfrac{1}{y_2} & x_1& \dfrac{x_1+y_2}{(y_2\oplus 1)x_2} \\[3mm]
\hline
&&&&\\[-3mm]
2& \dfrac{1\oplus y_2}{y_1y_2} & \dfrac{y_1}{y_1y_2\oplus y_2\oplus 1} & \dfrac{x_2y_1y_2 + y_2+ x_1}{(y_1y_2\oplus y_2\oplus 1)x_1x_2} & \dfrac{x_1+y_2}{(y_2\oplus 1)x_2} \\[3mm]
\hline
&&&&\\[-3mm]
3& \dfrac{1}{y_1y_2\oplus y_2} & \dfrac{y_1y_2\oplus y_1\oplus 1}{y_1} & \dfrac{x_2y_1y_2+y_2+x_1}{(y_1y_2\oplus y_2\oplus 1)x_1x_2} & \dfrac{x_2y_1+1}{x_1(y_1\oplus 1)} \\[3mm]
\hline
&&&&\\[-2mm]
4& y_1y_2\oplus y_2 &\dfrac{1}{y_1} & x_2 & \dfrac{x_2y_1+1}{x_1(y_1\oplus 1)} \\[3mm]
\hline
&&&&\\[-2mm]
5& y_2 & y_1 & x_2 & x_1\\[1mm]
\hline
\end{array}
\end{equation*}
\caption{Coefficients and cluster variables in type~$A_2$: general coefficients\label{A2seed}}
\end{table}

Therefore, we have
\begin{align*}
\Acal(B;t_0)=\ZZ\PP\left[x_1,x_2,\dfrac{x_1+y_2}{(y_2\oplus 1)x_2},\dfrac{x_2y_1y_2+y_2+x_1}{(y_1y_2\oplus y_2\oplus 1)x_1x_2}, \dfrac{x_2y_1+1}{x_1(y_1\oplus 1)}\right].
\end{align*}
\end{example}

Let $\Sigma_t=(\xx_t,\yy_t,B_t)$ be a seed and $\sigma$ a permutation on $\{1,\cdots,n\}$. We use $\sigma(\Sigma_t)$ to denote a new seed, which is defined as follows:
\begin{align*}
\sigma(\Sigma_t)&=(\sigma\xx_t,\sigma\yy_t,\sigma B_t),\\
    \sigma\xx_t&=(x_{\sigma(1);t},\dots,x_{\sigma(n);t}),\\
    \sigma\yy_t&=(y_{\sigma(1);t},\dots,y_{\sigma(n);t}),\\
    \sigma B_t&=(b'_{ij}),\quad b'_{ij}=b_{\sigma(i)\sigma(j);t}.
\end{align*}

The two seeds $\Sigma_t$ and $\Sigma_{t^\prime}$ are said to be \emph{equivalent} if there exists a permutation $\sigma$ such that $\Sigma_{t^\prime}=\sigma(\Sigma_t)$. The equivalence class of seeds defined by this equivalence relation is called a \emph{non-labeled seed}. A non-labeled seed with the representation is $\Sigma$ is denoted by $[\Sigma]$.
\begin{remark}\label{non-labeled-remark}
Notice that the permutation of a seed is compatible with mutations, that is, we have
\begin{align}
    \mu_{k}(\sigma(\Sigma_t))=\sigma\mu_{\sigma(k)}(\Sigma_t).\nonumber
\end{align}
This property makes non-labeled seed mutations well-defined.
\end{remark}

\begin{theorem} [\cite{cl2}*{Proposition 3 (ii)}] \label{non-labeled-cluster-thm}
Let $\Sigma_t$ and $\Sigma_{t^\prime}$ be two seeds of a cluster algebra $\mathcal A$. If there exists a permutation $\sigma$ such that $\xx_{t^\prime}=\sigma\xx_{t}$, then $\Sigma_{t^\prime}=\sigma(\Sigma_t)$.
\end{theorem}

By the aforementioned theorem, if $\xx_t=\xx_{t'}$, then we have $B_t=B_{t'}$ and $\yy_t=\yy_{t'}$. Therefore, the mutation of cluster $\mu_k(\xx)$ is well-defined. Moreover, if two clusters are identical as unordered sets (i.e., identical as \emph{non-labeled clusters}), then the two seeds are also identical as non-labeled seeds. Therefore, the mutation of non-labeled cluster $\mu_x([\xx])$ is also well-defined, where $x$ is a cluster variable in $[\xx]$.

The following proposition demonstrates that the coefficients does not affect combinatorial structure of clusters.

\begin{proposition}[\cite{cl2}*{Proposition 3 (i)}] \label{proindepen}
Let $\mathcal A_1$ be a cluster algebra with coefficient semifield $\mathbb P_1$ and $\mathcal A_2$ with coefficient semifield $\mathbb P_2$. Let $(\xx_t(k),\yy_t(k), B_t(k))$ be a seed of $\Acal_k$ at $t\in\TT_n$, $k= 1,2$. If $\Acal_1$ and $\Acal_2$ have the same initial exchange matrix at the rooted vertex $t_0$, then $x_{i;t}(1) =x_{j;t'}(1)$ if and only if $x_{i;t}(2) =x_{j;t'}(2)$, where $t, t'\in\TT_n$ and $i, j\in \{1,2,\cdots, n\}$.
\end{proposition}

\begin{definition}[Exchange graph]
The \emph{exchange graph} $\Gamma(B)$ of a cluster algebra $\Acal(B;t_0)$ is a graph whose vertices correspond to the non-labeled clusters of $\Acal(B;t_0)$ and whose edges correspond to the mutations of non-labeled clusters.
\end{definition}
According to Proposition \ref{proindepen}, we know that the exchange graph $\Gamma(B)$ of a cluster algebra $\mathcal A(B;t_0)$ is independent of the choice of the coefficient semifield. It only depends on the initial exchange matrix $B$.
\subsection{$c$-vectors, $g$-vectors and exchange quiver in cluster algebras} In this subsection, we recall the $c$-vectors, $g$-vectors and exchange quiver in cluster algebras.

\begin{definition}[Principal coefficient]
We state that a cluster pattern $v\mapsto \Sigma_v$ or a cluster algebra $\Acal(B;t_0)$ of rank $n$ has the \emph{principal coefficients} at the rooted vertex $t_0$ if $\mathbb{P}=\text{Trop}(y_1,\dots,y_n)$ and $\mathbf{y}_{t_0}=(y_1,\dots,y_n)$. In this case, we denote $\Acal=\Acal_\bullet(B;t_0)$.
\end{definition}
First, we define the $c$-vectors. For $\bb=(b_1,\dots,b_n)^{\mathrm{T}}$, we use the notation $[\bb]_+=([b_1]_+,\dots,[b_n]_+)^{\mathrm{T}}$, where $\mathrm{T}$ stands for transpose.
\begin{definition}[$c$-vector and $C$-matrix]
Let $\Acal_{\bullet}(B;t_0)$ be a cluster algebra with principal coefficients at $t_0$. We define the \emph{$c$-vector}\footnote{The letter ``$c$" of the $c$-vector is derived from "coefficients," not ``Coxeter."} $\cc_{j;t}$ as the degree of $y_i$ in $y_{j;t}$, that is, if $y_{j;t}=y_1^{c_{1j;t}}\cdots y_n^{c_{nj;t}}$, then
\begin{align}
\cc_{j;t}^{B;t_0}=\cc_{j;t}=\begin{bmatrix}c_{1j;t}\\ \vdots \\ c_{nj;t} \end{bmatrix}.
\end{align}
We define the \emph{$C$-matrix} $C_t^{B;t_0}$ as
\begin{align}
C_t^{B;t_0}:=(\cc_{1;t},\dots,\cc_{n;t}).
\end{align}
\end{definition}

\begin{remark}\label{c-recursion-remark}
The $c$-vectors are the same as those defined by the following recursion: For any $j \in\{1,\dots,n\}$,
\begin{align*}
\cc_{j;t_0}=\ee_j\quad \text{(canonical basis)},
\end{align*}
and for any $\begin{xy}(0,0)*+{t}="A",(10,0)*+{t'}="B",\ar@{-}^k"A";"B" \end{xy}$,
\begin{align*}
\cc_{j;t'} =
\begin{cases}
-\cc_{j;t} & \text{if $j=k$;} \\[.05in]
\cc_{j;t} + [b_{kj;t}]_+ \ \cc_{k;t} +b_{kj;t} [-\cc_{k;t}]_+
 & \text{if $j\neq k$.}
 \end{cases}
\end{align*}
This recursion demonstrates that the $C$-matrix can also be computed as follows: We set
$\tilde {B}_{t_0}:=\begin{bmatrix}B\\I_n\end{bmatrix}$, and define $\tilde{B}_{t'}=\mu_k(\tilde B_t)$ for any edge \begin{xy}(0,1)*+{t}="A",(10,1)*+{t'}="B",\ar@{-}^k"A";"B" \end{xy} of $\mathbb T_n$. The bottom $n\times n$ matrix of $\tilde B_t$ then coincides with $C^{B;t_0}_t$.
As the recursion formula only depends on exchange matrices,
we can regard $c$-vectors as vectors associated with $\TT_n$'s vertices.
\end{remark}

Now, we introduce the non-labeled version of $c$-vectors and $C$-matrices. To define them, we recall an important theorem:

\begin{theorem}[\cite{nak20}]\label{c-permutation-thm}
For any cluster algebra $\Acal(B;t_0)$ (not necessarily one with principal coefficients) and clusters $\xx_t,\xx_{t'}$ in $\Acal(B;t_0)$, if there exists a permutation $\sigma$ such that $\xx_{t^\prime}=\sigma\xx_{t}$, then $C^{B;t_0}_{t^\prime}=\sigma C^{B;t_0}_{t}=(\cc_{\sigma(1);t},\dots,\cc_{\sigma(n);t})$.
\end{theorem}

\begin{definition}[Non-labeled version of $c$-vector and $C$-matrix]\label{non-labeled-c}
For any cluster variable $x$ and non-labeled cluster $[\xx]$ containing $x$ in $\Acal(B;t_0)$, when $[\xx]=[\xx_t]$ and $x=x_{i;t}$, we define
\[\cc^{B;t_0}_{x,[\xx]}:=\cc_{i;t}^{B;t_0},\]
as the \emph{$c$-vector associated with $x$ in $[\xx]$}.
Furthermore, we define \[C^{B;t_0}_{[\xx]}:=\{\cc_{i;t}^{B;t_0}\mid 1\leq i \leq n\}\]
and refer to it as the \emph{$C$-matrix associated with $[\xx]$}.
\end{definition}
\begin{remark}
In Definition \ref{non-labeled-c}, the $C$-matrix associated with $[\xx]$ is a set. It may be considered as matrices if necessary, given the appropriate order. The $G$-matrix associated with $[\xx]$ that will be defined later follows the same rules.
\end{remark}
We note that $\cc^{B;t_0}_{x,[\xx]}$ and $C^{B;t_0}_{[\xx]}$ are well-defined by Proposition \ref{c-permutation-thm}, and that $\Acal(B;t_0)$ in Definition \ref{non-labeled-c} does not need to be a cluster algebra with principal coefficients according to Remark \ref{c-recursion-remark}.

The following theorem defines positive and negative signs for $c$-vectors.

\begin{theorem}[Sign-coherence, \cite{GHKK}]\label{thm:signs-ci}
Each column vector of a $C$-matrix is either a non-negative vector or a non-positive vector.
\end{theorem}

Consider incorporating a $c$-vector orientation along the edge of the exchange graph.

\begin{definition}[Green and red mutations] Let $\mathcal A(B; t_0)$ be a cluster algebra with initial exchange matrix $B$ at $t_0$.
A mutation of non-labeled cluster $\mu_x([\xx])$ in $\mathcal A(B; t_0)$ is called a {\em green mutation}, if $\cc^{B;t_0}_{x,[\xx]}$is a non-negative vector. Otherwise, it is called a {\em red mutation}.
\end{definition}

\begin{remark}
In \cite{hps18}, the green mutation is defined by using the sign of $\cc^{-B^T;t_0}_{i;t}$ instead of $\cc^{B;t_0}_{i;t}(=\cc^{B;t_0}_{x,[\xx]})$. This definition is equivalent to the current definition of the green mutation. It follows the equation \[C^{-B^T;t_0}_t=SC^{B;t_0}_tS^{-1}\] provided by \cite{nz}*{(2.7)}.
\end{remark}

\begin{definition}[Exchange quiver] Let $\mathcal A(B; t_0)$ be a cluster algebra with initial exchange matrix $B$ at $t_0$.
The {\em exchange quiver} of $\mathcal A(B; t_0)$ is the quiver $\overrightarrow{\Gamma}(B)$ whose vertices correspond to non-labeled clusters in $\mathcal A(B; t_0)$ and whose arrows correspond to green mutations.
\end{definition}

\begin{example}[Type $A_2$]\label{ex-exchange-quiver-A2}
We set $\PP=\{1\}$. The exchange quiver in the configuration of Example \ref{A2} is as follows:
 
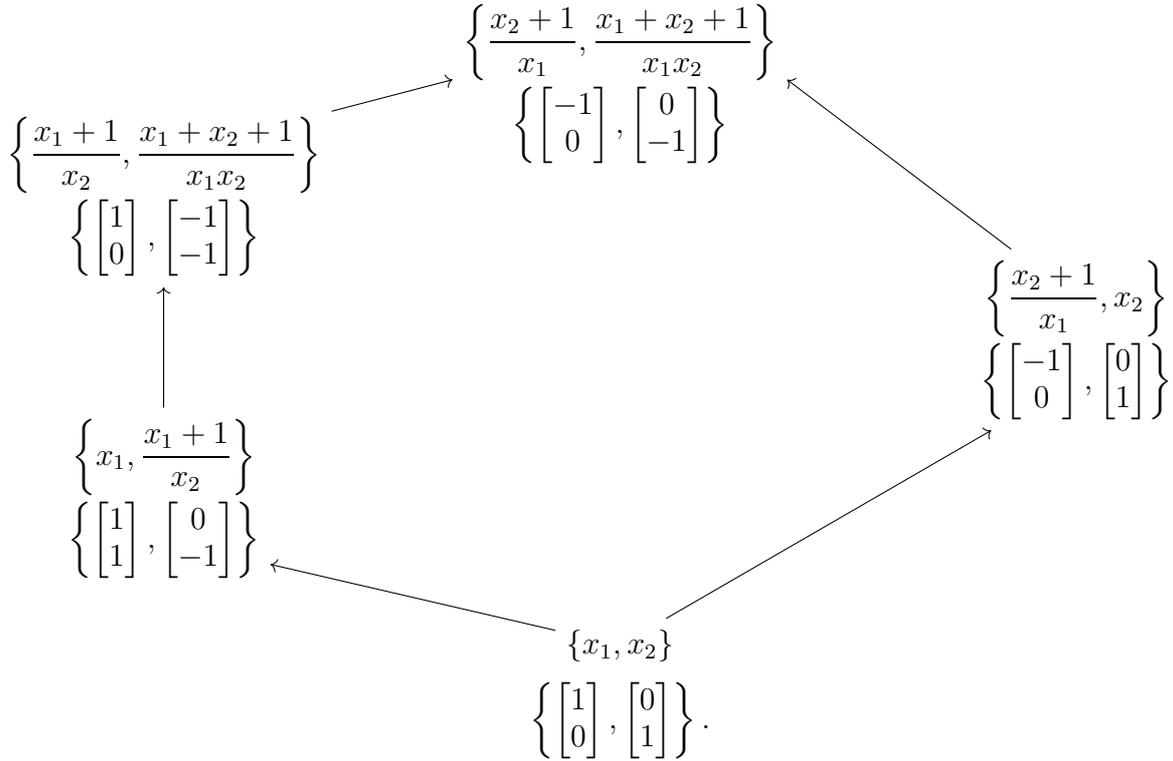
\[\begin{tikzpicture}
      \node (1) at (0,9) {$\left\{\dfrac{x_2+1}{x_1},\dfrac{x_1+x_2+1}{x_1x_2}\right\}$};
       \node (1') at (0,7.9) {$\left\{\begin{bmatrix}-1\\0\end{bmatrix},\begin{bmatrix}0\\-1\end{bmatrix} \right\}$};
      \node(2) at (-6,7.5) {$\left\{\dfrac{x_1+1}{x_2},\dfrac{x_1+x_2+1}{x_1x_2}\right\}$};
      \node (2') at (-6,6.4) {$\left\{\begin{bmatrix}1\\0\end{bmatrix},\begin{bmatrix}-1\\-1\end{bmatrix} \right\}$};
      \node (3) at (6,5.6) {$\left\{\dfrac{x_2+1}{x_1},x_2\right\}$};
      \node (3') at (6,4.5) {$\left\{\begin{bmatrix}-1\\0\end{bmatrix},\begin{bmatrix}0\\1\end{bmatrix} \right\}$};
      \node (4) at (-6,3.5) {$\left\{x_1,\dfrac{x_1+1}{x_2}\right\}$};
      \node (4') at (-6,2.4) {$\left\{\begin{bmatrix}1\\1\end{bmatrix},\begin{bmatrix}0\\-1\end{bmatrix} \right\}$};
      \node (5) at (0,1) {$\{x_1,x_2\}$};
      \node (5') at (0,0) {$\left\{\begin{bmatrix}1\\0\end{bmatrix},\begin{bmatrix}0\\1\end{bmatrix} \right\}.$};
      \draw[<-] (-2.2,8.5)--(2);
      \draw[<-] (2.2,8.5)--(3);
      \draw[<-] (2')--(4);
      \draw[<-] (3')--(5);
      \draw[<-] (4')--(5);
    \end{tikzpicture}\]
Here, the upper row for each vertex displays the non-labeled cluster, while the lower row displays the associated $C$-matrix.
\end{example}

Next, we will define the $g$-vector and $G$-matrix.
In fact, cluster variables can be viewed as homogeneous Laurent polynomials of a certain degree.
\begin{theorem}[\cite{fziv}*{Proposition 6.1}]\label{thmfziv}
Let $\Acal(B;t_0)$ be a cluster algebra with principal coefficients at seed $\Sigma_{t_0}$. Then each cluster variable $x_{k;t}$ is a homogeneous Laurent polynomial in $x_1,\dots,x_n,y_1,\dots,y_n$ under the following $\ZZ^n$-grading:
\begin{align}\label{grading}
\deg x_{i}=\ee_i,\quad \deg y_{i}=-\mathbf{b}_i,
\end{align}
where $\ee_i$ is the $i$th canonical basis of $\ZZ^n$ and $\mathbf{b}_i$ is the $i$th column vector of $B$.
\end{theorem}
\begin{definition}[$g$-vector and $G$-matrix]
Retain the notations in Theorem \ref{thmfziv}. The degree $\deg x_{k;t}$ of $x_{k;t}$ under the $\ZZ^n$-grading \eqref{grading} is called the \emph{$g$-vector} of $x_{k;t}$ and we denote it by $$\gg_{k;t}^{B;t_0}=\gg_{k;t}=\deg x_{k;t}.$$ The matrix $G_t^{B;t_0}=(\gg_{1;t},\dots,\gg_{n;t})$ is called a \emph{$G$-matrix} of $B$.
\end{definition}
\begin{remark}
The $g$-vectors are the same as those defined by the following recursion: For any $j\in\{1,\dots,n\}$,
\begin{align*}
\gg_{j;t_0}=\ee_j\quad \text{(canonical basis)},
\end{align*}
and for any $\begin{xy}(0,0)*+{t}="A",(10,0)*+{t'}="B",\ar@{-}^k"A";"B" \end{xy}$,
\begin{align}\label{g-recursion}
\gg_{j;{t'}}&=\begin{cases}
\gg_{j;t} \ \ & \text{if } j\neq k;\\
-\gg_{k;t}+\mathop\sum\limits_{i=1}^{n}[b_{ik;t}]_+\gg_{i;t}-\mathop\sum\limits_{i=1}^{n}[c_{ik;t}]_+\mathbf{b}_j &\text{if } j=k
\end{cases}
\end{align}
(see \cite{fziv}). As the recursion formula depends on exchange matrices,
we can regard $g$-vectors as vectors associated with vertices of $\TT_n$.
\end{remark}

According to the definition of $g$-vectors, the corresponding $g$-vector for each cluster variable can be determined in a cluster algebra with principal coefficients, but according to Proposition \ref{proindepen}, the same can be performed with coefficients that are not the principal ones.
Define the $g$-vectors and $G$-matrices associated with cluster variables and non-labeled clusters similarly to $c$-vectors and $C$-matrices.
\begin{definition}[Non-labeled version of $g$-vector and $G$-matrix]\label{non-labeled-g}
For any cluster variable $x$ in $\Acal(B;t_0)$, when $x=x_{i;t}$, we denote
\[\gg^{B;t_0}_{x}:=\gg_{i;t}^{B;t_0},\]
and refer to it as the \emph{$g$-vector associated with $x$}.
Furthermore, for any non-labeled cluster $[\xx]$ in $\Acal(B;t_0)$, when $[\xx]=[\xx_t]$, we denote
\[G^{B;t_0}_{[\xx]}:=\{\gg_{i;t}^{B;t_0}\mid 1\leq i \leq n\}\]
and refer to it as the \emph{$G$-matrix associated with $[\xx]$}.
\end{definition}

The following theorem is an important property about the relation between $C$-matrices and $G$-matrices:

\begin{theorem}[\cite{nz}]\label{c-g-duality}
For any exchange matrix $B$ and $t_0,t\in \TT_n$, we have
\[((G_t^{B;t_0})^\mathrm{T})^{-1}=SC_t^{B;t_0}S^{-1}.\]
\end{theorem}
\subsection{Quiver of $c$-clusters}
We define the quiver of $c$-clusters. The set of almost positive roots is denoted by $\Phi_{\geq -1}$, which is the union of the set $-\Delta$ of negative simple roots and the set $\Phi^+$ of positive roots. We define a map
\[\sigma_i\colon \Phi_{\geq -1} \to \Phi_{\geq -1},\quad \sigma_i(\alpha)=\begin{cases} \alpha \quad&\text{if $\alpha \in -\Delta\setminus \{-\alpha_i\}$}\\ s_i(\alpha)&\text{otherwise.}\end{cases}\]
This map is a self-bijection of $\apPhi$. For any Coxeter element $c=c_1\dots,c_n$, we define a self-bijection $\tau_c$ of $\apPhi$ as $\tau_c=\sigma_{c_1}\circ\sigma_{c_{2}}\circ\cdots\circ\sigma_{c_n}$. The following proposition is necessary to define the degree of $c$-compatibility.
\begin{proposition}[\cite{re-sp}*{Proposition 8.2}]
For any almost positive root $\alpha$ and Coxeter element $c$, there exists a non-negative integer $R$ such that $\tau_c^{-R}(\alpha)$ is a negative simple root.
\end{proposition}
\begin{definition}[$c$-compatible]
The \emph{$c$-compatibility degree} on the set of almost positive roots is the unique function
\[\apPhi \times \apPhi \to \ZZ_{\geq 0},\quad (\alpha, \beta) \mapsto (\alpha \parallel_c \beta)\]
characterized by the following two properties:
\begin{itemize}
    \item[(i)] $(-\alpha_i \parallel_c \beta) = \max\{0, b_i\}$, for all $i \in \{1,\dots,n\}$ and $\beta = \sum b_i\alpha_i \in \apPhi$
    \item[(ii)]$(\alpha \parallel_c \beta) = (\tau_c (\alpha) \parallel_c \tau_c(\beta))$, for all $\alpha,\beta\in\apPhi$.
\end{itemize}
Furthermore, if $(\alpha\parallel_c \beta)=(\beta\parallel_c \alpha)=0$, then we say that $\alpha$ and $\beta$ is \emph{$c$-compatible}.
\end{definition}
\begin{remark}
The degree of the $c$-compatibility is well-defined. It is proven by \cite{cp}*{Theorem 3.1}.
\end{remark}
A \emph{$c$-compatible subset} of $\apPhi$ is a set of pairwise $c$-compatible almost positive roots, and \emph{$c$-cluster} is the maximal $c$-compatible subset.
\begin{theorem}[\cites{fzii,ce-la-st}]\label{cluster-cardinality}
Let $\Phi$ be a Dynkin-type root system of rank $n$. We fix a Coxeter element $c$ of $W(\Phi)$.
\begin{itemize}
    \item [(1)] Each $c$-cluster of $\apPhi$ has a cardinality of $n$,
    \item [(2)] For any $c$-cluster $C$ and $\alpha \in C$, there exists a unique $\beta\in \apPhi$ such that $C\cup\{\beta\}\setminus\{\alpha\}$ is a $c$-cluster.
\end{itemize}
\end{theorem}
We will define a quiver whose vertex set is the set of $c$-clusters. For $\alpha\in \apPhi$, we set
\[R_c(\alpha)=\min\{n\in\ZZ_{\geq 0}\mid \tau_c^{-n}(\alpha)\in -\Delta\}.\]
\begin{definition}[Quiver of $c$-clusters]\label{c-cluster-lattice}
Let $\Phi$ be a Dynkin-type root system of rank $n$. We fix a Coxeter element $c$ of $W(\Phi)$. We consider the following quiver:
\begin{itemize}
    \item The vertex set of $\overrightarrow{\Gamma}(\Phi_{\geq -1},c)$ is the set of $c$-clusters;
    \item For $c$-clusters $C_1$ and $C_2$, if there exists $C'$ such that $C_1=C'\cup \{\alpha\}, C_2=C'\cup \{\beta\}$ and $R_c(\alpha)>R_c(\beta)$, then $C_1\to C_2$ in $\overrightarrow{\Gamma}(\Phi_{\geq -1},c)$. We refer to it as the \emph{quiver of $c$-clusters}.
\end{itemize}
\end{definition}

\begin{remark}
In Definition \ref{c-cluster-lattice}, it is impossible for $R_c(\alpha) = R_c(\beta)$. Indeed, assuming $R_c(\alpha) = R_c(\beta)$, we have
\[(\alpha\parallel_c\beta)=(\tau_c^{-R_c(\alpha)}(\alpha)\parallel_c\tau_c^{-R_c(\alpha)}(\beta))=(-\alpha_i\parallel_c -\alpha_j)=0\]
and $C'\cup\{\alpha,\beta\}$ is a $c$-compatible subset. This conflicts with Theorem \ref{cluster-cardinality} (1).
\end{remark}

\begin{example}[Type $A_2$]\label{c-cluster-lattice-A2}
Let $\Phi$ be a root system of $A_2$ type, and $c=s_2s_1$. Then we have
\[\Phi_{\geq -1}=\{\alpha_1,\alpha_2,\alpha_1+\alpha_2, -\alpha_1,-\alpha_2\}.\]
The number of orbits of $\tau_c^{-1}$ is 1, indeed, we have
\[
-\alpha_1\overset{\tau_c^{-1}}{\mapsto}\alpha_1\mapsto \alpha_2\mapsto-\alpha_2\mapsto \alpha_1+\alpha_2\mapsto-\alpha_1\mapsto\cdots
\]
The quiver of $c$-clusters is as follows:
 
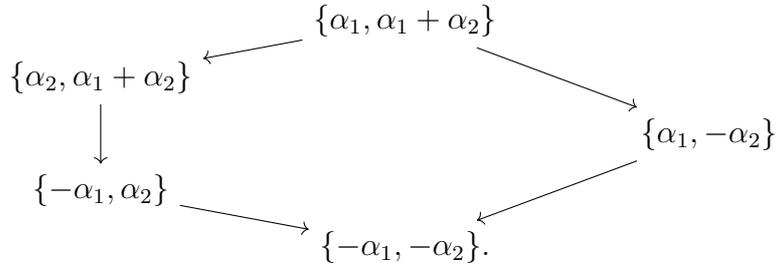
\[\begin{tikzpicture}
      \node (1) at (0,5) {$\{\alpha_1,\alpha_1+\alpha_2\}$};
      \node(2) at (-4,4.25) {$\{\alpha_2,\alpha_1+\alpha_2\}$};
      \node (3) at (4,3.5) {$\{\alpha_1,-\alpha_2\}$};
      \node (4) at (-4,2.75) {$\{-\alpha_1,\alpha_2\}$};
      \node (5) at (0,2) {$\{-\alpha_1,-\alpha_2\}.$};
      \draw[->] (1)--(2);
      \draw[->] (1)--(3);
      \draw[->] (2)--(4);
      \draw[->] (3)--(5);
      \draw[->] (4)--(5);
    \end{tikzpicture}\]
\end{example}
We can see that this quiver is anti-isomorphic to the one in Example \ref{ex-exchange-quiver-A2}.

\subsection{Support $\tau$-tilting quiver}
In this subsection, we recall $\tau$-tilting theory \cite{air}. For more information, refer to survey \cite{ir}. Let $K$ be a field. We fix a finite dimensional basic $K$-algebra $A$. The category of finitely generated right $A$-modules is denoted by $\mod A$ and the Auslander-Reiten translation in $\mod A$ is denoted by $\tau$.

For a full subcategory $\mathcal C$ of $\mod A$, we define the following subcategories.
\begin{itemize}
\item $\add\mathcal C$ is the additive closure of $\mathcal C$;
\item $\Fac\mathcal C$ consists of the factor modules of objects in $\add\mathcal C$;
\item $\mathcal C^\bot$ is the full subcategories given by
\begin{eqnarray}
 \mathcal C^\bot&:=&\{X\in \mod A| \Hom_{A}(\mathcal C, X)=0 \},\nonumber
\end{eqnarray}
\end{itemize}
Given a module $M\in\mod A$, we state that:
\begin{itemize}
    \item $M$ is a \emph{brick} if $\Hom_A(M,M)$ is a division ring;
    \item $M$ is \emph{$\tau$-rigid} if $\Hom_A(M,\tau M)=0$; and
    \item $|M|$ denotes the number of non-isomorphic indecomposable direct summands of $M$.
\end{itemize}

\begin{definition}[$\tau$-rigid and $\tau$-tilting pair]
Let $M$ be a module in $\mod A$ and $P$ be a projective module in $\mod A$. The pair $(M,P)$ is referred to as \emph{$\tau$-rigid} if $M$ is $\tau$-rigid and $\Hom_A(P,M)=0$.
It is referred to as \emph{$\tau$-tilting} (resp., \emph{almost $\tau$-tilting}) if it satisfies $|M|+|P|=n=|A|$ (resp., $|M|+|P|=n-1=|A|-1$). We state that $A$ is \emph{$\tau$-tilting-finite} if $\mod A$ has finitely large number of isomorphism classes of indecomposable $\tau$-rigid modules.
\end{definition}

Note that a module $M\in \mod A$ is called a \emph{support $\tau$-tilting module} if there exists a basic projective module $P$ such that $(M,P)$ is a $\tau$-tilting pair, or equivalently, if there exists an idempotent $e$ such that $M$ is a $\tau$-tilting $A/AeA$-module. In fact, such a basic projective module $P$ (or $eA$) is unique up to isomorphisms \cite{air}*{Proposition 2.3}.
A $\tau$-rigid pair $(M,P)$ is \emph{indecomposable} if $M\oplus P$ is indecomposable in $\mod A$, and it is \emph{basic} if both $M$ and $P$ are basic in $\mod A$. In the sequel, we always assume that $\tau$-rigid pairs and support $\tau$-tilting modules are basic and consider them up to isomorphisms.

\begin{definition}[Support $\tau$-tilting quiver]
The set of isomorphism classes of $\tau$-tilting pairs is denoted by $\stautilt A$. For $(M,P),(M',P')\in \stautilt A$, we set the arrow $(M,P)\to(M',P')$ if $\Fac M\supsetneq \Fac M'$ and there is no $(M'',P'')$ such that $\Fac M\supsetneq\Fac M''\supsetneq \Fac M'$. This quiver is denoted by $\overrightarrow{\Gamma}(\stautilt A)$ and is called the \emph{support $\tau$-tilting quiver of $A$}. 
\end{definition}

The key characterizations for two $\tau$-tilting pairs connected by an arrow on $\overrightarrow{\Gamma}(\stautilt A)$ are given below.

\begin{theorem}[\cite{air}*{Theorem 2.18}] \label{thmair}
Any basic almost $\tau$-tilting pair $(U,Q)$ in $\mod A$ is a direct summand of exactly two basic $\tau$-tilting pairs $(M,P)$ and $(M^\prime,P^\prime)$ in $\mod A$. Moreover, we have $$\{\Fac M,\Fac M^\prime\}=\{\Fac U,\prescript{\bot}{}{(\tau U)}\cap Q^\bot\}.$$
Particularly, either $\Fac M\subsetneq \Fac M^\prime$ or
$\Fac M^\prime\subsetneq\Fac M$ holds.
\end{theorem}

\begin{definition}[Left and right mutation]
Retain the notations in Theorem \ref{thmair}. We consider the operation $(M,P)\mapsto(M^\prime,P^\prime)$ a \emph{mutation} of $(M,P)$. If $\Fac M\subsetneq \Fac M^\prime$ holds, we call $(M^\prime, P^\prime)$ a {\em right mutation} of $(M,P)$. If $\Fac M^\prime \subsetneq \Fac M$ holds, we call $(M^\prime, P^\prime)$ a {\em left mutation} of $(M,P)$.
\end{definition}

\begin{theorem}[\cite{air}*{Theorem 2.33}] \label{mutation-charactorization}
For $\tau$-tilting pairs $(M,P)$ and $(M',P')$ the following conditions are equivalent.
\begin{itemize}
    \item[(i)] $(M,P)$ is a right mutation of $(M',P')$,
    \item[(ii)] $(M',P')$ is a left mutation of $(M,P)$,
    \item[(iii)] $\Fac M \supsetneq \Fac M'$ and there is no $\tau$-tilting pair $(M'',P'')$ such that $\Fac M \supsetneq \Fac M'' \supsetneq \Fac M'$.
\end{itemize}
\end{theorem}

Theorem \ref{mutation-charactorization} states that there exists an arrow $(M,P)\to (M',P')$ in $\overrightarrow{\Gamma}(\stautilt A)$ if and only if $(M',P')$ is a left mutation of $(M,P)$.

By definition of the support $\tau$-tilting module, we can consider the vertices of the support $\tau$-tilting quiver as support $\tau$-tilting modules instead of corresponding $\tau$-tilting pairs (this identification is used in Section \ref{gls-path-condition}). 

\begin{example}\label{ex:supp-tau-tilt-quiver-A2}
We set $A=K(1\leftarrow 2)$. Then $\overrightarrow{\Gamma}(\stautilt A)$ is as follows:
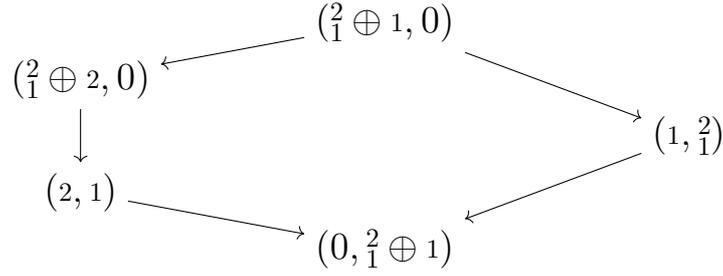
\[\begin{tikzpicture}
      \node (1) at (0,5) {\large $\left(\sst{2 \\ 1}\oplus \sst{1},0\right)$};
      \node(2) at (-4,4.25) {\large $\left(\sst{2 \\ 1}\oplus \sst{2},0\right)$};
      \node (3) at (4,3.5) {\large $\left(\sst{1},\sst{2\\1}\right)$};
      \node (4) at (-4,2.75) {\large$\left(\sst{2 },\sst{1}\right)$};
      \node (5) at (0,2) {\large $\left(0,\sst{2 \\ 1}\oplus \sst{1}\right)$};
      \draw[->] (1)--(2);
      \draw[->] (1)--(3);
      \draw[->] (2)--(4);
      \draw[->] (3)--(5);
      \draw[->] (4)--(5);
    \end{tikzpicture}\]
We can see that this quiver is isomorphic to the one in Example \ref{c-cluster-lattice-A2}.
\end{example}
\subsection{GLS path algebra}
We introduce the GLS path algebra, which is defined by \cite{gls-i}. We fix a Dynkin-type root system $\Phi$ and a Coxeter element $c\in W(\Phi)$. Let $D_{\Phi}=\text{diag}(d_1,\dots,d_n)$ be a symmetrizer of the Cartan matrix $C_\Phi$, i.e., a matrix such that $D_\Phi C_\Phi$ is symmetric, where each $d_i$ is a positive integer. Clearly, $D_{\Phi}$ is not unique for each $C_\Phi$. We will define the \emph{GLS path algebra} $H(C_\Phi,D_{\Phi},c)$. Let $Q^c$ represent the following quiver:
\begin{align*}
    (Q^c)_0&=\{1,\dots,n\},\\
    (Q^c)_1&=\{\alpha_{ij}\colon i\to j\mid c=\cdots s_i\cdots s_j\cdots,\ s_is_j\neq s_js_i\}\cup\{\varepsilon_i\colon i \to i\mid 1\leq i\leq n\}.
\end{align*}
Let $I$ represent an ideal $KQ^c$ generated by the following relations:
\begin{itemize}
    \item [(H1)] For any $i\in (Q^c)_0$, $\varepsilon_i^{d_i}=0$,
    \item[(H2)] For any $\alpha_{ij} \in (Q^c)_1$, $\varepsilon_i^{|C_{ji}|}\alpha_{ij}=\alpha_{ij}\varepsilon_j^{|C_{ij}|}$.
\end{itemize}
We define the GLS path algebra associated with $(C_\Phi,D_{\Phi},c)$ as follows: \[H:=H(C_\Phi,D_{\Phi},c):=KQ^c/I.\]
It is a finite dimensional $K$-algebra. Moreover, if $\Phi$ is simply-laced and $D_\Phi=I_n$, then $H$ is isomorphic to a representation-finite path algebra for $\Phi$ and $c$.

Let $\{e_1,\dots,e_n\}$ be the idempotents associated with the $Q^c$ vertices. We define $H_i:=e_iHe_i$ and $M_i=Me_i$ for $M\in \mod H$.
\begin{definition}[Locally free]
For $M\in \mod H$, $M$ is \emph{locally free} if $M_i$ is a free $H_i$-module for each $i$.
\end{definition}

Let $M\in \mod H$ be a locally free module. For each $i$, the rank of $M_i$ as a $H_i$ module is denoted by $r_i$. We denote this vector as $\underline{\text{rank}}_HM:=(r_1,\dots,r_n)$ and refer to it as the \emph{rank vector} of $M$.

Listed below are some of the properties of $H$ demonstrated by Gei\ss, Leclerc, and Schr\"oer that are required in this paper.

\begin{theorem}[(1):\cite{gls-i}*{Theorem 1.2}, (2):\cite{gls-reg}*{Corollary 6.3}]\label{facts-generalized-path-alg}\noindent
\begin{itemize}
    \item [(1)]The algebra $H$ is $1$-Iwanaga-Gorenstein, i.e., $\mathrm{inj. dim~}H\leq 1$ and $\mathrm{proj. dim~}DH\leq 1$. Moreover, the following are equivalent for $M\in \mod H$:
    \begin{itemize}
        \item [(i)] $\mathrm{proj. dim~} M\leq 1$,
        \item [(ii)] $\mathrm{inj. dim~} M\leq 1$,
        \item [(iii)] $\mathrm{proj. dim~} M< \infty$,
        \item [(iv)] $\mathrm{inj. dim~} M< \infty$,
        \item [(v)] $M$ is locally free,
    \end{itemize}
    \item[(2)] $M\in \mod H$ is $\tau$-rigid (resp. $\tau$-tilting) if and only if $M$ is rigid and locally free (resp. tiltng).
\end{itemize}
\end{theorem}

The notation $H$ shall represent $H(C_\Phi,D_{\Phi},c)$ wherever used as the symbol for an algebra in the remainder of this paper.

\section{Three graph isomorphisms}
\subsection{Graph isomorphism $\tilde{\theta}_c$}
The graph isomorphism between the quiver of $c$-clusters and a certain exchange quiver is provided in \cite{cp}. We fix a Dynkin-type root system $\Phi$ and a Coxeter element $c\in W(\Phi)$. We obtain a skew-symmetrizable matrix $B^c=(b_{ij})$ from the Cartan matrix $C_\Phi=(C_{ij})$ and $c$ as follows:
\[b_{ij}=\begin{cases}
0 & \quad \text {if $i=j$},\\
C_{ij} & \quad \text {if $c=\cdots s_j\cdots s_i\cdots$},\\
-C_{ij} & \quad \text {if $c=\cdots s_i\cdots s_j\cdots$}.
\end{cases}\]

We note that $-B^c=B^{c^{-1}}$.

\begin{theorem}[\cites{cp}]
We fix a Dynkin-type root system $\Phi$ of rank $n$ and a Coxeter element $c\in W(\Phi)$. Let $\Xcal(B^c)$ represent the set of cluster variables in $\Acal(B^c;t_0)$. Then the map
\[\theta_c\colon \Xcal(B^c)\to \apPhi,\quad \dfrac{f(x_1,\dots,x_n)}{x_1^{d_1}\cdots x_n^{d_n}}\mapsto \sum_{i=1}^n d_i\alpha_i\]
is well-defined and bijective, where $f(x_1,\dots,x_n)$ is a polynomial in $n$ variables with $\mathbb{ZP}$ coefficients. Moreover, this map induces the bijection between the set of clusters of $\Acal(B^c;t_0)$ and the set of $c$-clusters in $\apPhi$.
\end{theorem}

From construction the exchange quiver of $B^c$ and the quiver of $c$-clusters, the following corollary is clearly demonstrated:

\begin{corollary}\label{graph-iso-phi}
 The map $\theta_c$ induces the graph isomorphism $\tilde\theta_c \colon \overrightarrow{\Gamma}(B^c)^{\mathrm{op}}\to\overrightarrow{\Gamma}(\apPhi,c)$, where $\overrightarrow{\Gamma}(B^c)^{\mathrm{op}}$ is the opposite quiver of $\overrightarrow{\Gamma}(B^c)$.
\end{corollary}

\begin{remark}
Corollary \ref{graph-iso-phi} holds even if $\tilde\theta_c$ is a map from $\overrightarrow{\Gamma}(B^c)$ to $\overrightarrow{\Gamma}(\apPhi,c)$ since the graph of $\overrightarrow{\Gamma}(B^c)$ is the same as one of $\overrightarrow{\Gamma}(B^c)^{\mathrm{op}}$. The reason for using $\overrightarrow{\Gamma}(B^c)^{\mathrm{op}}$ is to show later that $\tilde\theta_c$ is a quiver isomorphism.
\end{remark}

\subsection{Graph isomorphisms $\tilde{\phi}_c$ and $\tilde{\psi}_c$}
Next, we will explain graph isomorphisms between the support $\tau$-tilting quiver and the exchange quiver, as well as the support $\tau$-tilting quiver and the quiver of $c$-clusters. We begin with the following fact provided by \cite{gls-v}:
\begin{theorem}[\cite{gls-v}*{Theorem 1.2 (c)}]\label{psi+}
The correspondence $M \mapsto \sum_{i=1}^n r_i\alpha_i$ induces a bijection between isomorphism classes of indecomposable $\tau$-rigid modules in $H=H(C_\Phi,D_\Phi,c)$ and roots in $\Phi^+$. Particularly, $H$ is $\tau$-tilting finite.
\end{theorem}
The bijection from Theorem \ref{psi+} is denoted by $\phi^+_c$. We set $\psi^+_c:=\theta_c^{-1}\circ \phi^+_c$. Then, $\psi^+_c$ has the following bijective correspondence:
\[\psi^+_c\colon \ind (\textrm{$\tau$-\textsf{rigid}}\ H) \to \Xcal^+(B^c),\quad M\mapsto \dfrac{f(x_1,\dots,x_n)}{x_1^{\mathrm{rank}_{H_1 }M_1}\cdots x_n^{\mathrm{rank}_{H_n} M_n}},\]
where $\Xcal^+(B^c)=\Xcal(B^c)\setminus\{x_1,\dots,x_n\}$ and $\ind (\textrm{$\tau$-\textsf{rigid}}\ H)$ is the set of isomorphism classes of indecomposable $\tau$-rigid modules of $H$.

Gei\ss, Leclerc and Schr\"oer presented the following theorem about $\psi^+_c$;

\begin{theorem}[\cite{gls-v}*{Theorem 1.2 (d)}]\label{psi+bijection}
The map $\psi^+_c$ induces a bijection $\tilde\psi^+_c$ between the isomorphism classes of basic support $\tau$-tilting modules of $H$ and the maximal subsets consisting of non-initial cluster variables of non-labeled clusters in $\Acal(B^c;t_0)$.
\end{theorem}

From this theorem, we have the following corollary:

\begin{corollary}\label{theta+bijection}
The map $\phi^+_c$ induces a bijection $\tilde\phi^+_c$ between the isomorphism classes of basic support $\tau$-tilting modules of $H$ and the maximal subsets consisting of positive roots of $c$-clusters.
\end{corollary}

We will demonstrate that Theorem \ref{psi+bijection} and Corollary \ref{theta+bijection} provide graph isomorphisms for $\overrightarrow{\Gamma}(\stautilt H)$, $\overrightarrow{\Gamma}(\apPhi,c)$ and $\overrightarrow{\Gamma}(B^c)^{\mathrm{op}}$. The bijection $\phi^+_c$ is extended as follows:
\begin{align*}
\phi_c\colon \ind (\textrm{$\tau$-\textsf{rigid}-\textsf{pair}}\ H) \to \apPhi,\quad (M,0)\mapsto \phi^+_c(M), \quad (0,P_i) \mapsto- \alpha_i,
\end{align*}
where $P_i=e_iH$ and $\ind (\textrm{$\tau$-\textsf{rigid}-\textsf{pair}}\ H)$ is the set of isomorphism classes of indecomposable $\tau$-rigid pairs of $H$. We assign $\psi_c:=\theta_c^{-1}\circ \phi_c$. Then, $\psi_c$ has the following correspondence:
\begin{align*}\psi_c\colon \ind (\textrm{$\tau$-\textsf{rigid}-\textsf{pair}}\ H) \to \Xcal(B^c),\quad (M,0)\mapsto \psi^+_c(M), \quad (0,P_i) \mapsto x_i.
\end{align*}

The following lemma holds for any $K$-algebra $A$:
\begin{lemma}\label{proj-M_i}
The module $P_i$ is a direct summand of $P$ for any $\tau$-tilting pair $(M,P)$ of $A$ if and only if $M_i=0$.
\end{lemma}

\begin{proof}
Since $\Hom_A(P,M)=0$, we have $\Hom_A(P_i,M)=0$ if $P_i$ is a direct summand of $P$. Then, $\Hom_A(P_i,M)\simeq M_i=0$ holds true. Conversely, if $M_i=0$ and $P_i$ is not a direct summand of $P$, then $\Hom_A(P\oplus P_i,M)=0$ and $(M,P\oplus P_i)$ is $\tau$-rigid. It conflicts with the maximality of the $\tau$-tilting pair. Therefore, $M_i=0$ implies $P_i$ is a direct summand of $P$.
\end{proof}

\begin{theorem}\label{graphiso-theta}
The map $\phi_c$ induces a bijection $\tilde\phi_c$ between isomorphism classes of basic support $\tau$-tilting pairs and $c$-clusters. Particularly, $\tilde\phi_c\colon\overrightarrow{\Gamma}(\stautilt H)\to\overrightarrow{\Gamma}(\apPhi,c)$ is a graph isomorphism.
\end{theorem}

\begin{proof}
We fix a $\tau$-tilting pair $(M,P)=\bigoplus( N_i, Q_i)$ of $H$, where $(N_i,Q_i)$ are indecomposable. Corollary \ref{theta+bijection} states that the maximal subset of positive roots of a certain $c$-cluster $C$ is $\bigcup_i\phi_c(N_i,0)$. Due of the bijectivity of $\tilde\phi^+$, it is sufficient to demonstrate that $\bigcup_i\phi_c(0,Q_i)=C\setminus\bigcup_i\phi_c(N_i,0)$. According to Lemma \ref{proj-M_i}, $\bigcup\phi_c (0,Q_i)$ is the set of negative versions of the simple roots that do not appear in linear representations by simple roots of all $\phi_c(N_i,0)$. By definition of the compatibility degree and $c$-cluster, this set coincides with $C\setminus\bigcup_i\phi_c(N_i,0)$.
\end{proof}
Since $\tilde\theta_c$ is a graph isomorphism, we have the following corollary:
\begin{corollary}\label{graphiso-psi}
The map $\psi_c$ induces a bijection $\tilde\psi_c$ between isomorphism classes of basic support $\tau$-tilting pairs and clusters in $\Acal(B^c;t_0)$. Particularly, $\tilde\psi_c\colon \overrightarrow{\Gamma}(\stautilt H)\to \overrightarrow{\Gamma}(B^c)^{\mathrm{op}}$ is a graph isomorphism.
\end{corollary}
Now, we have three graph isomorphisms that appear in the commutative diagram below:
\[\begin{xy}
   (-10,0)*{\overrightarrow{\Gamma}(B^c)^{\mathrm{op}}}="A",(30,15)*{\overrightarrow{\Gamma}(\stautilt H) }="B",(30,-15)*{\overrightarrow{\Gamma}(\apPhi,c)}="D",\ar^{\tilde\psi_c}@{<-} "A";"B"\ar_{\tilde\theta_c}@{->} "A";"D"\ar^{\tilde\phi_c}@{->}"B";"D" 
\end{xy}\]

\begin{remark}\label{previous-study-remark}
There are previous studies discussing the relation between exchange quivers and support $\tau$-tilting quivers of finite dimensional hereditary algebras. Corollary \ref{graphiso-psi} can be recovered from these studies via the isomorphism between the support $\tau$-tilting quiver of a GLS path algebra and that of the corresponding finite dimensional hereditary algebra. The graph isomorphism between the support $\tau$-tilting quiver of a finite dimensional hereditary algebra and the corresponding exchange quiver of cluster algebras is provided in \cite{rupel15} (though the preprint \cite{Hubery} appears to be the first to mention this, but the link to this article seems to be broken at the moment). The graph isomorphism between the support $\tau$-tilting quivers of a GLS path algebra and the corresponding finite dimensional hereditary algebra is given in \cite{gls-reg}*{Theorem 1.1}. 
\end{remark}

\section{$\tilde\psi_c$ is quiver isomorphism}
In this section, we prove the following theorem:
\begin{theorem}\label{lattice-op-iso-psi}
The graph isomorphism $\tilde\psi_c$ is a quiver isomorphism.
\end{theorem}

This theorem is demonstrated by the fact that the sign of the $c$-vector of a cluster variable and the $c$-vector introduced in the corresponding $\tau$-rigid pair are inverted. The absolute values of the two $c$-vectors do not generally coincide, but their dual, the $g$-vector, is known to coincide up to a factor of $-1$. The consistency of $g$-vectors have been essentially demonstrated in \cite{gls-v}, it cannot be applied in its original form in this work, so it is modified into a form that is applicable before use.

\subsection{$C$-matrix and $G$-matrix of $\tau$-tilting pair}
To prove Theorem \ref{lattice-op-iso-psi}, we introduce the $G$-matrices and the $C$-matrices of $\tau$-tilting pairs. For $M\in\mod A$, we take the minimal projective presentation of $M$
\[\xymatrix{\bigoplus\limits_{i=1}^nP_i^{b_i}\ar[r]&\bigoplus\limits_{i=1}^nP_i^{a_i}\ar[r]&M\ar[r]&0}.\]
The column vector $\mathbf{ g}(M):=(a_1-b_1,\cdots,a_n-b_n)^\mathrm{T}\in\mathbb Z^n$ is called the \emph{$g$-vector} of $M$.
For a $\tau$-rigid pair $(M,P)$, its \emph{$g$-vector} is defined as $$\mathbf{ g}_{(M,P)}:=\mathbf{ g}(M)-\mathbf{ g}(P) \in \ZZ^n.$$

A basic $\tau$-tilting pair $(M,P)$ is expressed using the form $(M,P)=\bigoplus\limits_{i=1}^n(N_i,Q_i)$ to indicate that each pair $(N_i,Q_i)$ is an indecomposable $\tau$-rigid pair.

For a basic $\tau$-tilting pair $(M,P)=\bigoplus\limits_{i=1}^n(N_i,Q_i)$, its \emph{$G$-matrix} is the following integer matrix:

\[G_{(M,P)}:=(\mathbf{ g}_{(N_1,Q_1)},\cdots,\mathbf{ g}_{(N_n,Q_n)}).\] The following theorem demonstrates that $G_{(M,P)}$ is a regular matrix.
\begin{theorem}[\cite{air}*{Theorem 5.1}] Let $(M,P)=\bigoplus\limits_{i=1}^n(M_i,Q_i)$ be a basic $\tau$-tilting pair in $\mod A$. Then, $\mathbf{ g}_{(M_1,Q_1)},\cdots,\mathbf{ g}_{(M_n,Q_n)}$ constitute the $\mathbb Z$-basis of $\mathbb Z^n$.
\end{theorem}
Next, we define the $c$-vectors and the $C$-matrices. We begin with the following theorem:
\begin{theorem}[\cites{asai}]\label{thmtref}
Let $(M,P)$ be a basic $\tau$-tilting pair in $\mod A$. Suppose the left mutation $(M^\prime,P^\prime)$ of $(M,P)$ is obtained by replacing an indecomposable $\tau$-rigid pair $(N,Q)$ with $(N',Q')$. Then, there exists a unique brick $\tilde{S}_{(N,Q),(M,P)}$ up to isomorphisms such that $\tilde{S}_{(N,Q),(M,P)}$ belongs to $\Fac M\cap (\Fac M')^\perp$.
\end{theorem}

Let $(M,P)$ be a basic $\tau$-tilting pair in $\mod A$. Suppose the mutation $(M^\prime,P^\prime)$ of $(M,P)$ is obtained by replacing an indecomposable $\tau$-rigid pair $(N,Q)$ with $(N',Q')$. We define the \emph{brick $S_{(N,Q),(M,P)}$ associated with $(N,Q)$ in $(M,P)$} as:
\[S_{(N,Q),(M,P)}=\begin{cases}
    \tilde{S}_{(N,Q),(M,P)}&\text{if $(M',P')$ is a left mutation of $(M,P)$,}\\
    \tilde{S}_{(N',Q'),(M',P')} &\text{otherwise.}
\end{cases}\]

We define the \emph{$c$-vector associated with $(N,Q)$ in $(M,P)$} as:
\[\cc_{(N,Q),(M,P)}=\begin{cases}
\underline{\dim}_KS_{(N,Q),(M,P)} &\text{if $(M',P')$ is a left mutation of $(M,P)$,}\\
-\underline{\dim}_KS_{(N,Q),(M,P)} &\text{otherwise.}
\end{cases}\]
Moreover, when $(M,P)=\bigoplus\limits_{i=1}^n(N_i,Q_i)$, its \emph{$C$-matrix} is the following integer matrix
\[C_{(M,P)}:=(\mathbf{ c}_{(N_1,Q_1),(M,P)},\dots,\mathbf{ c}_{(N_n,Q_n),(M,P)}).\]

By definition of the sign of the $c$-vector and Theorem \ref{mutation-charactorization}, we have the following lemma:

\begin{lemma}[\cite{cgy}*{Lemma 4.5}]\label{lemright}
Let $(M,P)$ be a basic $\tau$-tilting pair in $\mod A$. Suppose the mutation $(M^\prime,P^\prime)$ of $(M,P)$ is obtained by replacing an indecomposable $\tau$-rigid pair $(N,Q)$ with $(N',Q')$. Then the following statements are equivalent:
\begin{itemize}
\item[(i)] $(M^\prime, P^\prime)$ is a left mutation of $(M,P)$;
\item[(ii)] $(M,P)\to (M',P')$ in $\overrightarrow{\Gamma}(\stautilt A)$;
\item[(iii)] $\mathbf{ c}_{(N,Q),(M,P)}>0$.
\end{itemize}
\end{lemma}

The following theorem is the $\tau$-tilting version of Theorem \ref{c-g-duality}:

\begin{theorem}[\cite{asai}*{Section 3.4}]\label{c-g-duality-tilting}
For any basic $\tau$-tilting pair $(M,P)=\bigoplus_i(N_i,Q_i)$ be in $\mod A$, we define $D$ (resp. $D'$) as the diagonal matrix so that its $(i, i)$ entry is $\dim_K \End_A(e_i(A/\mathsf{rad} A))$ (resp. $\dim_K \End_A S_{(N_i,Q_i),(M,P)}$). Then we have 
\[(G_{(M,P)}^\mathrm{T})^{-1}=DC_{(M,P)}(D')^{-1}.\]
\end{theorem}

\subsection{Proof of Theorem \ref{lattice-op-iso-psi}}
In this subsection, for a $\tau$-tilting pair $(M,P)=\bigoplus(N_i,Q_i)$, we regard the $G$-matrix $G^{B^c;t_0}_{\tilde\psi_c(M,P)}$ as a matrix
\[(\mathbf{ g}_{\psi_c(N_1,Q_1)},\dots,\mathbf{ g}_{\psi_c(N_n,Q_n)}).\] Similarly, we regard the $C$-matrix $C^{B^c;t_0}_{\tilde\psi_c(M,P)}$ as a matrix
\[(\mathbf{ c}_{\psi_c(N_1,Q_1),\tilde\psi_c(M,P)},\dots,\mathbf{ c}_{\psi_c(N_n,Q_n),\tilde\psi_c(M,P)}).\]

The proof of Theorem \ref{lattice-op-iso-psi} employs an important theorem on $g$-vectors from \cite{gls-v}. Since there are some differences in the notations in \cite{gls-v} and those in this paper, we will start by restating the theorem using our notations. From this subsection, we assume the semifield $\PP$ of cluster algebras is trivial, that is, $\PP=\{1\}$ (According to Proposition \ref{proindepen}, $\overrightarrow{\Gamma}(B)$ is independent of $\PP$). For the sake of readability, let $\xx^{B;t_0}_t$ denote $\xx_t$ in $\Acal(B;t_0)$. We have the following proposition:

\begin{proposition}\label{same-graph}
For any skew-symmetrizable matrix $B$ and for any $t\in \TT_n$, we have $\xx_t^{B;t_0}=\xx_t^{-B;t_0}$. Particularly, $\Xcal(B)=\Xcal(-B)$ and the exchange graph ${\Gamma}(B)$ is identical to ${\Gamma}(-B)$.
\end{proposition}
\begin{proof}
It arises from the fact that $-\mu_k(B)=\mu_k(-B)$ and the symmetry of cluster mutation recurrences \eqref{eq:x-mutation} when $\PP=\{1\}$.
\end{proof}

According to this proposition, the notation $\gg^{-B^c;t_0}_{\psi^+_{c}(M)}$ makes sense because the cluster variable set of $\Acal(B^c;t_0)$ and $\Acal(-B^c;t_0)$ are identical.

The following fact is provided in \cite{gls-v}.
\begin{theorem}[\cite{gls-v}*{Proposition 7.2}]\label{g-matrix-correspondence-gls}
We consider the minimal injective presentation for any indecomposable $\tau$-rigid module $M\in \mod H$.
\[\xymatrix{0\ar[r]&M\ar[r]&\bigoplus\limits_{i=1}^nI_i^{a_i}\ar[r]&\bigoplus\limits_{i=1}^nI_i^{b_i}}.\]
Then, the column vector $\hat{\gg}(M):=(b_1-a_1,\dots,b_n-a_n)$ coincides with $\gg^{-B^c;t_0}_{\psi^+_{c}(M)}$.
\end{theorem}

\begin{remark}
The notation in Proposition \ref{g-matrix-correspondence-gls} may seem strange. This is due to the fact that the skew-symmetrizable matrix corresponding to the Coxeter element $c$ in \cite{gls-v} is $-B^c(=B^{c^{-1}})$ instead of $B^c$. Furthermore, $\hat{\gg}(M)$ is referred to as the \emph{$g$-vector} in \cite{gls-v}.
\end{remark}

Theorem \ref{g-matrix-correspondence-gls} examines the relation between minimal injective presentation and $g$-vectors in cluster algebra, but in our paper, the $g$-vector of a module is defined by its minimal projective presentation. Therefore, we would like to apply the Theorem \ref{g-matrix-correspondence-gls} after applying the standard dual $D(-):=\Hom_K(-,K)$ to the minimal projective presentation. To accomplish this, however, the standard dual must preserve the $\tau$-rigidity of $M$. The following lemma provides proof for it.

\begin{lemma}\label{DM-tau-tilting}
For any Dynkin-type root system $\Phi$ and any Coxeter element $c\in W(\Phi)$, if $M\in \mod H$ is $\tau$-rigid (resp. support $\tau$-tilting), then $DM\in \mod H^{\mathrm{op}}$ is also $\tau$-rigid (resp. support $\tau$-tilting).
\end{lemma}
\begin{proof}
We only prove the $\tau$-tilting part (the $\tau$-rigid part is proved in the same way). By definition of support $\tau$-tilting, there exists an idempotent $e$ such that $M$ is a $\tau$-tilting module of $H/HeH$-module. Since there exist a Dynkin-type root system $\Phi'$ and a Coxeter element $c'\in W(\Phi')$ such that $H/HeH\simeq H(C_{\Phi'},D_{\Phi'},c')$, we can assume that $M$ is a $\tau$-tilting module. According to Theorem \ref{facts-generalized-path-alg} (2), $M$ is tilting. We will prove that $DM$ is a tilting module of $H^{\textrm{op}}$. By considering the group of extensions, $\Ext^1_H(M,M)=0$ if and only if $\Ext^1_{H^{\textrm{op}}}(DM,DM)=0$. Moreover, according to Theorem \ref{facts-generalized-path-alg} (1) and $H^\mathrm{op}\simeq H(C_\Phi,D_\Phi,c^{-1})$, $\mathrm{proj.dim} M\leq 1$ if and only if $\mathrm{proj.dim} DM\leq 1$. Furthermore, as $|H|=|H^{\mathrm{op}}|$ and $|M|=|DM|$, we have that $|H|=|M|$ if and only if $|H^{\mathrm{op}}|=|DM|$. Therefore, $DM$ is a tilting module of $H^{\mathrm{op}}$. According to Theorem \ref{facts-generalized-path-alg} (2) again, $DM$ is a $\tau$-tilting module.
\end{proof}

\begin{lemma}\label{minus-dual}
We have the following commutative diagram:
\[\begin{xy}
   (0,5)*{\Xcal^+(B^c)}="B",(60,5)*{\ind (\textrm{$\tau$-$\mathsf{rigid}$}\ H)}="C",(0,-10)*{\Xcal^+(-B^c)=\Xcal^+(B^{c^{-1}})}="D",(60,-10)*{\ind (\textrm{$\tau$-$\mathsf{rigid}$}\ H^{\mathrm{op}})
   }="E",\ar^{\hspace{-5mm}\psi^+_c}@{<-} "B";"C"\ar_{\psi^+_{c^{-1}}}@{<-}"D";"E",\ar@{=}"B";"D",\ar^{D(-)}@{->}"C";"E"
\end{xy}\]
Particularly, if $(\psi^+_c)^{-1}(x)=N$, then we have $(\psi^+_{c^{-1}})^{-1}(x)=DN$.
\end{lemma}
\begin{proof}
By definition of $\psi^+_c$, if $x=\dfrac{f(x_1,\dots,x_n)}{x^{d_1}\cdots x_n^{d_n}}$, then
\[\underline{\mathrm{rank}}_H((\psi^+_c)^{-1}(x))=\underline{\mathrm{rank}}_{H^{\mathrm{op}}}((\psi^+_{c^{-1}})^{-1}(x))=(d_1,\dots,d_n).\]
By Lemma \ref{DM-tau-tilting}, $DN$ is an indecomposable $\tau$-rigid module. As $\underline{\mathrm{rank}}_HN=\underline{\mathrm{rank}}_{H^{\mathrm{op}}}DN$ and the bijectivity of $\psi^+_{c}$ , we have $(\psi^+_{c^{-1}})^{-1}(x)=DN$.
\end{proof}

The following theorem is our notation version of Theorem \ref{g-matrix-correspondence-gls}.

\begin{theorem}\label{g-matrix-correspondence}
For any indecomposable $\tau$-rigid pair $(N,Q)$ in $H$, we have $\gg_{(N,Q)}=-\gg^{B^c;t_0}_{\psi_c(N,Q)}$. Particularly, for any basic $\tau$-tilting pair $(M,P)$ in $H$, we have $G_{(M,P)}=-G^{B^c;t_0}_{\tilde\psi_c(M,P)}$.
\end{theorem}

\begin{proof}
The claim is followed by the definition of the $g$-vector of $(0,Q)$ if $N=0$. We assume that $N\neq 0$ (and $Q=0$). Then we have $\gg_{(N,0)}=\gg(N)$. By definition of $g$-vector, we have $\gg(N)=-\hat\gg(DN)$. By Lemma \ref{DM-tau-tilting}, $DN$ is an indecomposable $\tau$-rigid module of $H^{\textrm{op}}$. Therefore, by Theorem \ref{g-matrix-correspondence-gls}, we have $\hat\gg(DN)=\gg^{B^c;t_0}_{\psi^+_{c^{-1}}(DN)}$. By Lemma \ref{minus-dual}, we have $\gg^{B^c;t_0}_{\psi^+_{c^{-1}}(DN)}=\gg^{B^c;t_0}_{\psi^+_{c}(N)}$. From the above discussion, we have $\gg_{(N,0)}=-\gg^{B^c;t_0}_{\psi_{c}(N,0)}$.
\end{proof}

By using Theorem \ref{g-matrix-correspondence}, we have the following corollary:
\begin{corollary}\label{c-matrix-correspondence}
We fix any basic $\tau$-tilting pair $(M,P)=\bigoplus\limits_{i=1}^n(N_i,Q_i)$ in $H$. If necessary, rearrange the order of the indices in $(N_i,Q_i)$ such that the order is obtained by applying mutations from $((P_1,0),(P_2,0),...,(P_n,0))$. Then, we have \[C_{(M,P)}=-D^{-1}SC^{B^c;t_0}_{\tilde\psi_c(M,P)}S^{-1}D',\] where $S$ is the skew-symmetrizer of $B^c$ and $D,D'$ are matrices in Theorem \ref{c-g-duality-tilting}. In particular, for any indecomposable $\tau$-rigid pair $(N,Q)$ and basic $\tau$-tilting pair $(M,P)$ contains $(N,Q)$ in $H$, $\cc_{(N,Q),(M,P)}>0$ if and only if $\cc^{B^c;t_0}_{\psi_c(N,Q),\tilde\psi_c(M,P)}<0$.
\end{corollary}
\begin{proof}
It is inferred from Theorems \ref{g-matrix-correspondence}, \ref{c-g-duality-tilting}, and \ref{c-g-duality}.
\end{proof}

We are ready to prove Theorems \ref{lattice-op-iso-psi}.

\begin{proof}[Proof of Theorem \ref{lattice-op-iso-psi}]
By Corollary \ref{c-matrix-correspondence} and Lemma \ref{lemright}, $(M,P)\to (M',P')$ in $\overrightarrow{\Gamma}(\stautilt H)$ if and only if $\tilde{\psi}_c(M',P')\to \tilde{\psi}_c(M,P)$ in $\overrightarrow{\Gamma}(B^c)$. Therefore, $\tilde{\psi}_c\colon \overrightarrow{\Gamma}(\stautilt H)\to \overrightarrow{\Gamma}(B^c)^{\mathrm{op}}$ is a quiver isomorphism.
\end{proof}

\section{Relation between quivers of one algebra and its opposite}
In this section, we compare $\overrightarrow{\Gamma}(\stautilt H)$ to $\overrightarrow{\Gamma}(\stautilt H^{\mathrm{op}})$, and $\overrightarrow{\Gamma}(B^c)$ to $\overrightarrow{\Gamma}(-B^c)$. This observation is key in proving that $\tilde{\theta}_c$ is a quiver isomorphism.  
\subsection{Relation between support $\tau$-tilting quivers of $H$ and $H^{\mathrm{op}}$}\label{gls-path-condition}

In this subsection, we regard vertices of the support $\tau$-tilting quiver as support $\tau$-tilting modules instead of corresponding $\tau$-tilting pairs. In other words, we identify a $\tau$-tilting pair $(M,P)$ with the corresponding support $\tau$-tilting module $M$.

From the discussion in Sections 3 and 4, we obtain the following facts:

\begin{proposition}\label{graph-iso-dual}
Let $H$ be a GLS path algebra. The standard duality $D\colon \mod H\to \mod H^{\mathrm{op}}$ induces the graph isomorphism $\tilde{D}\colon \overrightarrow{\Gamma}(\stautilt H)\to \overrightarrow{\Gamma}(\stautilt H^{\mathrm{op}})$.
\end{proposition}

\begin{proof}
We have a composition of graph isomorphisms
\[\overrightarrow{\Gamma}(\stautilt H)\xrightarrow{\tilde{\psi}_c} \overrightarrow{\Gamma}(B^c)=\overrightarrow{\Gamma}(-B^c)\xrightarrow{\tilde{\psi}_{c^{-1}}^{-1}} \overrightarrow{\Gamma}(\stautilt H^\mathrm{op})\] by Corollary \ref{graphiso-psi} and Proposition \ref{same-graph}. By Lemma \ref{minus-dual}, this graph isomorphism is induced by the standard duality $D$.
\end{proof}

As demonstrated by Proposition \ref{graph-iso-dual}, $\overrightarrow{\Gamma}(\stautilt H)$ and $\overrightarrow{\Gamma}(\stautilt H^{\mathrm{op}})$ are graph-isomorphic, but not quiver-isomorphic. If an arrow $M\to M'$ exists in $\overrightarrow{\Gamma}(\stautilt H)$ and $M'\subsetneq M$, then we have $DM\to DM'$ in $\overrightarrow{\Gamma}(\stautilt H^{\mathrm{op}})$ as $DM'\subsetneq DM$. What about the other arrow directions? The subsequent theorem is its solution. The arrows in $\overrightarrow{\Gamma}(\stautilt H)$ and $\overrightarrow{\Gamma}(\stautilt H^{\mathrm{op}})$ are all opposite to each other.

\begin{theorem}\label{condition-tilting-true}
Let $H$ be a GLS path algebra. If there exists an arrow $M\to M'$ in $\overrightarrow{\Gamma}(\stautilt H)$ and $|M|=|M'|$ holds, then there is an arrow $DM'\to DM$ in $\overrightarrow{\Gamma}(\stautilt H^{\mathrm{op}})$.
\end{theorem}
Let us look at a concrete example.
\begin{example}
We set $H=K(1\leftarrow 2)$ (this is the same situation as Example \ref{ex:supp-tau-tilt-quiver-A2}), then we have $H^{\mathrm{op}}=K(1\rightarrow 2)$. Set $D\colon \mod H\to \mod H^\mathrm{op}$. Then, we have
\[D(\sst{2 \\ 1})=\sst{1 \\ 2},\ D(\sst{2})=\sst{2},\ D(\sst{1})=\sst{1}.\]
The quiver $\overrightarrow{\Gamma}(\stautilt H^{\mathrm{op}})$ is as follows:
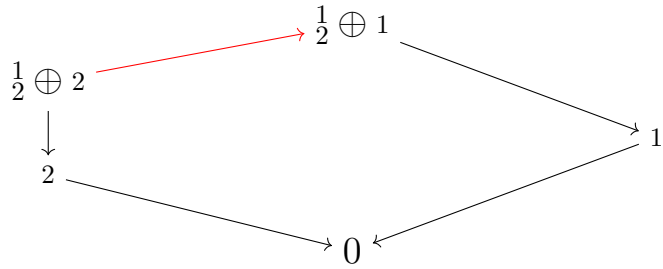
\[\begin{tikzpicture}
      \node (1) at (0,5) {\large $\sst{1 \\ 2}\oplus \sst{1}$};
      \node(2) at (-4,4.25) {\large $\sst{1 \\ 2}\oplus \sst{2}$};
      \node (3) at (4,3.5) {\large $\sst{1}$};
      \node (4) at (-4,3) {\large$\sst{2}$};
      \node (5) at (0,2) {\large $0$};
      \draw[<-,red] (1)--(2);
      \draw[->] (1)--(3);
      \draw[->] (2)--(4);
      \draw[->] (3)--(5);
      \draw[->] (4)--(5);
    \end{tikzpicture}\]
Compared to the quiver $\overrightarrow{\Gamma}(\stautilt H^{\mathrm{op}})$ in Example \ref{ex:supp-tau-tilt-quiver-A2} (replace the vertex $(M,P)$ with $M$), it can be seen that if there exists an arrow $M\to M'$ in $\overrightarrow{\Gamma}(\stautilt H)$ and $|M|=|M'|$ holds, then there is an arrow $DM'\to DM$ in $\overrightarrow{\Gamma}(\stautilt H^{\mathrm{op}})$ (indicated by the red arrow).

\end{example}

We will prove Theorem \ref{condition-tilting-true}. The following lemma is provided by Happel-Unger \cite{haun3} for hereditary algebras, and it is possible to eliminate the heredity assumption. For examples of proofs, see \cite{eno-sak}*{Theorem A.7}.
\begin{lemma}[\cite{haun3}]\label{ex-seq-equiv}
Let $L\in \mod A$ be an almost complete tilting module. Suppose $L$ has two indecomposable tilting complements $X,Y$, where $X\not\simeq Y$. The following are equivalent:
\begin{itemize}
    \item [(i)] $\Fac(L\oplus X)\supsetneq\Fac(L\oplus Y)$,
    \item [(ii)] There exists an exact sequence
    \[0\rightarrow X \xrightarrow{f} \tilde{L} \xrightarrow{g} Y\rightarrow 0,\]
    where $\tilde L\in \add L$.
\end{itemize}
\end{lemma}

\begin{proof}[Proof of Theorem \ref{condition-tilting-true}]
Let $M,M'\in \mod H$ be basic support $\tau$-tilting modules, and suppose there is an arrow $M\to M'$ in $\overrightarrow{\Gamma}(\stautilt H)$ and that $|M|=|M'|$ holds. We note that $M$ and $M'$ share the common support part $P$, that is, $(M,P)$ and $(M',P)$ are both $\tau$-tilting pairs. Now, $M$ and $M'$ can be assumed to be $\tau$-tilting for the same reasons in the proof of Lemma \ref{DM-tau-tilting}. Then, there exists an almost complete module $L$ and indecomposable modules $X$ and $Y$ such that $M\simeq L\oplus X$ and $M'\simeq L\oplus Y$. By definition of the support $\tau$-tilting quiver, we have $\Fac M\supsetneq \Fac M'$. According to Theorem \ref{facts-generalized-path-alg} (2), $M$ and $M'$ are tilting modules. Thus, according to Lemma \ref{ex-seq-equiv}, there exists an exact sequence
    \[0\rightarrow X \xrightarrow{f} \tilde{L} \xrightarrow{g} Y\rightarrow 0,\]
    where $\tilde L\in \add L$. Applying the standard duality $D(-)$ to this sequence, we have
    \[0\rightarrow DY \xrightarrow{Dg} D\tilde{L} \xrightarrow{Df} DX\rightarrow 0,\]
where $D\tilde L\in \add DL$.    
As $DM$ and $DM'$ are tilting modules, $DX$ and $DY$ are two tilting complements of $DL$. Therefore, by Lemma \ref{ex-seq-equiv}, we have $\Fac DM\subsetneq \Fac DM'$. Therefore, we have $DM'\to DM$ in $\overrightarrow{\Gamma}(\stautilt H^{\mathrm{op}})$. This concludes the proof.
\end{proof}

\subsection{Relation between exchange quivers of $B^c$ and $-B^c$}
Using the quiver isomorphism $\tilde{\psi}_c$, the relation between $\overrightarrow{\Gamma}(B^c)$ and $\overrightarrow{\Gamma}(-B^c)$ is the same as that between $\overrightarrow{\Gamma}(\stautilt H)$ and $\overrightarrow{\Gamma}(\stautilt H^{\mathrm{op}})$.

\begin{proposition}\label{green-initial}
If $x$ in $\Acal(B;t_0)$ is initial, then $\cc^{B;t_0}_{x,[\xx]}>0$ for any $[\xx]$ that contains $x$.
\end{proposition}

If $B=B^c$ in Proposition \ref{green-initial}, it is analogous to the fact that if $M$ and $M'$ are connected by an arrow in $\overrightarrow{\Gamma}(\stautilt H)$ and $M'\subsetneq M$, we have an arrow $M\to M'$ in $\overrightarrow{\Gamma}(\stautilt H)$. This is demonstrated by applying the quiver isomorphism $\tilde{\psi}_c$ to this observation. In the following lemma, we prove that Proposition \ref{green-initial} holds without making any assumptions about $B$ or addressing $\tau$-tilting theory.

\begin{lemma}[\cite{nz}*{(1.13)}]\label{c-g-duality2}
For any exchange matrix $B$ and $t_0,t\in \TT_n$, we have
\[C_t^{B;t_0}=(G_{t_0}^{B_t^{\mathrm{T}};t})^{\mathrm{T}}.\]
\end{lemma}

\begin{lemma}[\cite{cl2}*{Theorem 10}]\label{connected-graph}
For any cluster algebra $\Acal(B;t_0)$, non-labeled clusters of $\Acal(B;t_0)$ containing a set $X$ consisting of certain cluster variables form a connected subgraph of $\Gamma(B)$.
\end{lemma}

\begin{proof}[Proof of Proposition \ref{green-initial}]
If $x=x_i$, there exists a cluster $\xx_t$ and mutation sequence $\mu$ excluding $\mu_i$ such that $\xx_t=\mu(\xx_{t_0})$ (in particular, $x_i$ is the $i$th component of $\xx_t$) and $[\xx]=[\xx_t]$ by applying Lemma \ref{connected-graph} with $X=\{x_i\}$. Then $\gg^{B_t^T;t}_{i;t_0}=\ee_i$. According to Lemma \ref{c-g-duality2}, the $i$th component of $\cc^{B;t_0}_{i;t}=\cc^{B;t_0}_{x,[\xx]}$ must be $1$. Therefore, $\cc^{B;t_0}_{x,[\xx]}>0$. 
\end{proof}
As an analog of Theorem \ref{condition-tilting-true} in $\overrightarrow{\Gamma}(\stautilt A)$, we have the following theorem:
\begin{theorem}\label{condition-true}
We fix any Dynkin-type root system $\Phi$ and any Coxeter element $c\in W(\Phi)$. For a non-initial cluster variable $x$ and any cluster $[\xx]$ containing $x$ in $\Acal(B^c;t_0)$, $\cc^{B^c;t_0}_{x,[\xx]}>0$ implies that $\cc^{-B^c;t_0}_{x,[\xx]}<0$.
\end{theorem}
In other words, each position that is a green mutation that replaces a non-initial cluster variable with another in $\overrightarrow{\Gamma}(B)$ are red mutation in $\overrightarrow{\Gamma}(-B)$.  Using Proposition \ref{green-initial} and Theorem \ref{condition-true}, we can determine all signs of $c$-vectors in $\Acal(-B^c;t_0)$ from those in $\Acal(B^c;t_0)$.\begin{example}[Type $A_2$]
When $B=\begin{bmatrix}0&1\\-1&0\end{bmatrix}$, we have the following exchange quiver:
 
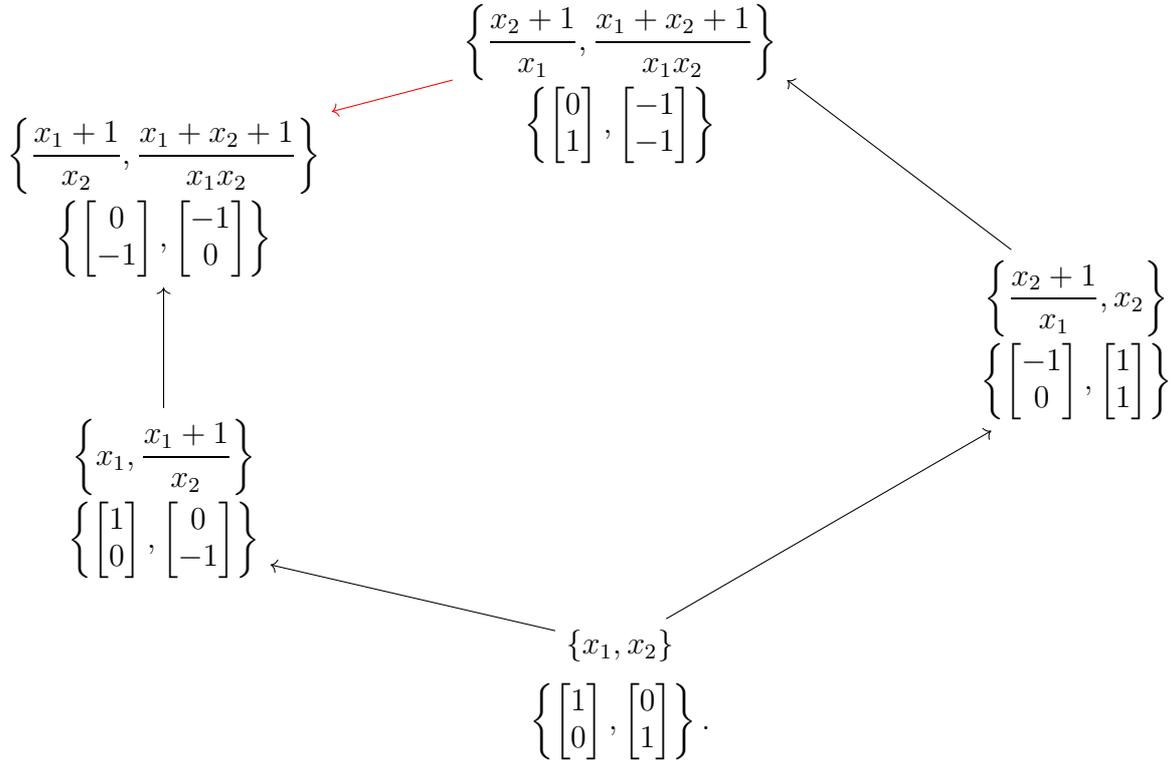
\[\begin{tikzpicture}
      \node (1) at (0,9) {$\left\{\dfrac{x_2+1}{x_1},\dfrac{x_1+x_2+1}{x_1x_2}\right\}$};
       \node (1') at (0,7.9) {$\left\{\begin{bmatrix}0\\1\end{bmatrix},\begin{bmatrix}-1\\-1\end{bmatrix} \right\}$};
      \node(2) at (-6,7.5) {$\left\{\dfrac{x_1+1}{x_2},\dfrac{x_1+x_2+1}{x_1x_2}\right\}$};
      \node (2') at (-6,6.4) {$\left\{\begin{bmatrix}0\\-1\end{bmatrix},\begin{bmatrix}-1\\0\end{bmatrix} \right\}$};
      \node (3) at (6,5.6) {$\left\{\dfrac{x_2+1}{x_1},x_2\right\}$};
      \node (3') at (6,4.5) {$\left\{\begin{bmatrix}-1\\0\end{bmatrix},\begin{bmatrix}1\\1\end{bmatrix} \right\}$};
      \node (4) at (-6,3.5) {$\left\{x_1,\dfrac{x_1+1}{x_2}\right\}$};
      \node (4') at (-6,2.4) {$\left\{\begin{bmatrix}1\\0\end{bmatrix},\begin{bmatrix}0\\-1\end{bmatrix} \right\}$};
      \node (5) at (0,1) {$\{x_1,x_2\}$};
      \node (5') at (0,0) {$\left\{\begin{bmatrix}1\\0\end{bmatrix},\begin{bmatrix}0\\1\end{bmatrix} \right\}.$};
      \draw[red,->] (-2.2,8.5)--(2);
      \draw[<-] (2.2,8.5)--(3);
      \draw[<-] (2')--(4);
      \draw[<-] (3')--(5);
      \draw[<-] (4')--(5);
    \end{tikzpicture}\]
Compared to the exchange quiver in Example \ref{ex-exchange-quiver-A2}, it can be seen that only the direction of the arrow corresponding to the mutation in which another non-initial variable is exchanged for a non-initial variable is different (indicated by the red arrow).
\end{example}

By using the quiver isomorphism $\tilde{\psi_c}$, we will prove Theorem \ref{condition-true}.

\begin{proof}[Proof of Theorem \ref{condition-true}]
Let $x$ be a non-initial cluster variable and $[\xx]$ a cluster in $\Acal(B^c;t_0)$. We assume that $\cc^{B^c;t_0}_{x,[\xx]}>0$. Then, there is an arrow $[\xx]\to [\xx']$ in $\overrightarrow{\Gamma}(B^c)$, where $[\xx']=\mu_x([\xx])$. As $\tilde\psi_c$ is a quiver isomorphism, we have an arrow $\tilde\psi_c^{-1}([\xx'])\to \tilde\psi_c^{-1}([\xx])$ in $\overrightarrow{\Gamma}(\stautilt H)$. Let $x'$ be a cluster variable in $[\xx']$ that is exchanged with $x$ by mutation $\mu_x\colon[\xx]\mapsto[\xx']$. Now, $x'$ is also non-initial since $\cc^{B^c;t_0}_{x',[\xx']}<0$ (if $x'$ is initial, then it conflicts with Proposition \ref{green-initial}). Therefore, we have $\left|\tilde\psi_c^{-1}([\xx])\right|=\left|\tilde\psi_c^{-1}([\xx'])\right|$. According to Theorem \ref{condition-tilting-true} and Lemma \ref{minus-dual}, we have an arrow $\tilde\psi_{c^{-1}}^{-1}([\xx])\to \tilde\psi_{c^{-1}}^{-1}([\xx'])$ in $\overrightarrow{\Gamma}(\stautilt H^{\mathrm{op}})$. Therefore, we have an arrow $[\xx']\to [\xx]$ in $\overrightarrow{\Gamma}(-B^c)$ and $\cc^{-B^c;t_0}_{x,[\xx]}<0$.
\end{proof}

\section{$\tilde\theta_c$ is quiver isomorphism}
In this section, our goal is the following theorem:
\begin{theorem}\label{lattice-op-iso}
The graph isomorphism $\tilde\theta_c$ is a quiver isomorphism.
\end{theorem}
We will prove Theorem \ref{lattice-op-iso} by using facts provided in the previous section. By abusing the notation for a cluster variable $x\in\Acal(B^c;t_0)$, we denote
\[\tau_c(x):=\theta_c^{-1}\tau_c(\theta_c(x)),\quad R_c(x):=R_c(\theta_c(x)).\]
It allows us to introduce the orientation of the $c$-cluster quiver to the exchange graph. In other words, we set $[\xx] \to [\xx']$ if $R(x)>R(x')$ for $[\xx]=X\cup\{x\},\ [\xx']=X\cup\{x'\}$. According to Theorem \ref{lattice-op-iso}, this quiver and the exchange quiver are inverted with respect to each other, i.e.,
\begin{align}\label{exchange-R}
[\xx] \to [\xx']\ \text{in}\ \overrightarrow{\Gamma}(B^c)\Leftrightarrow R(x)<R(x').
\end{align}
Furthermore, \eqref{exchange-R} can be rephrased by the following lemma:

\begin{lemma}\label{green-preserve-lemma}
The following condition is equivalent to \eqref{exchange-R}:
\begin{align}
&\text{For $[\xx]=X\cup\{x\},\ [\xx']=X\cup\{x'\}$ and $[\xx]\to[\xx']$ in $\overrightarrow{\Gamma}(B^c)$ with $x$ non-initial,}\label{green-preservation} \\&\text{then there exists an arrow $\tau^{-1}_c([\xx])\to \tau^{-1}_c([\xx'])$.}\nonumber
\end{align}
\end{lemma}
We refer to \eqref{green-preservation} the \emph{green-preservation condition}.

\begin{proof}[Proof of Lemma \ref{green-preserve-lemma}]
We assume that \eqref{exchange-R} holds for any $[\xx],[\xx']$. If $x$ is non-initial, then $x'$ is also non-initial (if not so, it conflicts with Proposition \ref{green-initial}). Thus $R(\tau_c^{-1}(x))=R(x)-1$ and $R(\tau_c^{-1}(x'))=R(x')-1$ hold, and we have $R(\tau_c^{-1}(x))<R(\tau_c^{-1}(x'))$. Therefore we have an arrow $\tau^{-1}_c([\xx])\to \tau^{-1}_c([\xx'])$ in $\overrightarrow{\Gamma}(B^c)$. Next, we assume that \eqref{exchange-R} does not hold for certain $[\xx],[\xx']$. Choose $[\xx],[\xx']$ such that $[\xx]\to [\xx']$ and $R(x)>R(x')$. If $x$ was initial, then \eqref{exchange-R} is always true. Therefore $x$ is non-initial. Moreover, $x'$ is also non-initial by Proposition \ref{green-initial}. Then, if the green-preservation condition was true, as $x,\tau^{-1}_c(x),\dots,\tau_c^{-R(x')+1}(x)$ are all non-initial, we have an arrow $\tau^{-R(x')}_c([\xx])\to\tau^{-R(x')}_c([\xx'])$. However, $\tau^{-R(x')}_c(x')$ is initial, and this conflicts with Proposition \ref{green-initial}.
\end{proof}

Now, we recall the interpretation of $\tau_c$ by mutations provided in \cite{cp}.
\begin{proposition}[\cite{cp}]\label{tau-interpretation}
We assume that $x\in\Acal(B^c;t_0)$ is the $k$th variable in $\mu(\xx_{t_0})$, where $\mu$ is a mutation sequence. Then, $\tau_c(x)$ (resp. $\tau_c^{-1}(x)$) is the $k$th variable of $\mu\circ\mu_{c_n}\circ\cdots\circ\mu_{c_1}(\xx_{t_0})$ (resp. $\mu\circ\mu_{c_1}\circ\cdots\circ\mu_{c_n}(\xx_{t_0})$).
\end{proposition}

The following lemma is important.

\begin{lemma}\label{c-mat-trick}
For any cluster variable $x$ of $\Acal(B^c;t_0)$ and any non-labeled cluster $[\xx]$ containing $x$, we have the following equation:
\[\cc^{B^c;t_0}_{\tau_c^{-1}(x),\tau_c^{-1}([\xx])}=-\cc^{-B^c;t_0}_{x,[\xx]}.\]
Particularly, $\cc^{B^c;t_0}_{\tau_c^{-1}(x),\tau_c^{-1}([\xx])}>0$ if and only if $\cc^{-B^c;t_0}_{x,[\xx]}<0$.
\end{lemma}
\begin{proof}
We have $\mu_{c_1}\circ\cdots\circ\mu_{c_n}(B^c)=B^c$ and the $C$-matrix of $\mu_{c_1}\circ\cdots\circ\mu_{c_n}(\xx_{t_0})$ is $-I_n$, because $\mu_{c_n},\dots, \mu_{c_1}$ in this sequence are all sink mutations. Therefore, we have \[\mu_{c_1}\circ\cdots\circ\mu_{c_n}(\tilde{B}^c_{t_0})=\begin{bmatrix}B^c\\-I_n\end{bmatrix}=-(\widetilde{-B}^c_{t_0}).\]
As $\mu(-A)=-\mu(A)$ for any mutation sequence $\mu$, we inductively conclude that
\begin{align}\label{extended-eq}
    \mu\circ\mu_{c_1}\circ\cdots\circ\mu_{c_n}(\tilde{B}^c_{t_0})=-\mu(\widetilde{-B}^c_{t_0}).
\end{align}   
According to Remark \ref{c-recursion-remark} and Proposition \ref{tau-interpretation}, if $\xx=\mu(\xx_{t_0})$, the bottom $n\times n$ matrix of the left-hand side matrix of \eqref{extended-eq} coincides with $C_{\tau^{-1}_c([\xx])}^{B^c;t_0}$. On the other hand, the bottom $n\times n$ matrix of the right-hand side matrix of \eqref{extended-eq} coincides with $-C_{[\xx]}^{-B^c;t_0}$ according to Proposition \ref{same-graph}. Therefore, we have $C_{\tau^{-1}_c([\xx])}^{B^c;t_0}=-C_{[\xx]}^{-B^c;t_0}$ and this concludes the proof.
\end{proof}

Now, we are ready to prove Theorem \ref{lattice-op-iso}.

\begin{proof}[Proof of Theorem \ref{lattice-op-iso}]
We prove that the green-preservation condition holds. We set $[\xx]=X\cup\{x\},\ [\xx']=X\cup\{x'\}$ and $[\xx]\to[\xx']$ in $\overrightarrow{\Gamma}(B^c)$ and assume that $x$ is non-initial. Then we have $\cc^{B^c;t_0}_{x,[\xx]}>0$. According to Theorem \ref{condition-true}, we have $\cc^{-B^c;t_0}_{x,[\xx]}<0$. According to Lemma \ref{c-mat-trick}, we have $\cc^{B^c;t_0}_{\tau_c^{-1}(x),\tau_c^{-1}([\xx])}>0$. Therefore, we have an arrow $\tau^{-1}_c([\xx])\to \tau^{-1}_c([\xx'])$.
\end{proof}
\section{$\tilde{\phi}_c$ is quiver isomorphism}

Next, we discuss the isomorphism $\tilde{\phi}_c$.

\begin{theorem}\label{lattice-iso-phi}
The graph isomorphism $\tilde{\phi}_c$ is a quiver isomorphism.
\end{theorem}

This theorem follows from Theorems \ref{lattice-op-iso-psi} and \ref{lattice-op-iso}. However, since the representation-theoretic interpretation of $\tau_c$ is important, we provide a more direct proof of Theorem \ref{lattice-iso-phi} using it, rather than relying on cluster algebra theory.

\subsection{Interpreting the transformation $\tau_c$ in terms of $\tau$-tilting theory of GLS-path algebras}
We denote $\tau^{-}:=\mathrm{Tr}D$. We say that $M\in \mod A$ is a \emph{$\tau^{-}$-rigid module} if $\Hom_A(\tau^{-}M,M)=0$.

\begin{definition}[$\tau^-$-rigid and $\tau^-$-tilting pair]
Let $M$ be a module in $\mod A$ and $I$ be an injective module in $\mod A$. The pair $(M,I)$ is referred to as \emph{$\tau^{-}$-rigid} if $M$ is $\tau^{-}$-rigid and $\Hom_A(M,I)=0$. It is referred to as \emph{$\tau^{-}$-tilting} (resp., \emph{almost $\tau^{-}$-tilting}) if it satisfies $|M|+|I|=n=|A|$ (resp., $|M|+|I|=n-1=|A|-1$).
\end{definition}

We note that $(M,I)$ is $\tau^-$-rigid (resp., $\tau^-$-tilting) if and only if $(DM,DI)$ is $\tau$-rigid (resp., $\tau$-tilting). 

\begin{theorem}[\cite{air}]
Let $A$ be a $K$-algebra. Then we have a bijection
\[\stautilt A\to \stautiltminus A,\ (M,P)\mapsto(\tau M\oplus \nu P,\nu M_{\mathrm{pr}}),\]
where $\nu$ is the Nakayama functor and $M_{\mathrm{pr}}$ is the maximal projective summand of $M$.
\end{theorem}

If $A$ is a GLS path algebra $H$, then we have the following lemma:

\begin{lemma}\label{tau-tau-equivalent}
For $M\in \mod H$, $M$ is a $\tau$-rigid module if and only if $M$ is a $\tau^{-}$-rigid module.
\end{lemma}

\begin{proof}
We assume that $M$ is $\tau^{-}$-rigid. Then, $DM$ is $\tau$-rigid. By Lemma \ref{DM-tau-tilting}, we have that $M$ is $\tau$-rigid. The converse case is proved in the same way.
\end{proof}

We remark that this fact is also pointed out in \cite{gls-i}*{Corollary 3.6}.

By Lemma \ref{tau-tau-equivalent}, $(\tau M\oplus \nu P, M_{\mathrm{pr}})$ is a $\tau$-tilting pair.

In fact, the transformation $(M,P)\mapsto (\tau M\oplus \nu P, M_{\mathrm{pr}})$ is the ``$\tau_c$-transformation" in $\overrightarrow{\Gamma}(\stautilt H)$. We will show it by the following theorem. For a $c$-cluster $\{\beta_1,\dots,\beta_n\}$, we denote by
\[\tau_c(\{\beta_1,\dots,\beta_n\}):=\{\tau_c(\beta_1),\dots,\tau_c(\beta_n)\}.\]

\begin{theorem}\label{description-tauc-rep}
  We have the following:
 \begin{itemize}
     \item [(1)] For an indecomposable $\tau$-rigid pair $(N,0)$, we have \[\tau_c(\phi_c(N,0))=\begin{cases}
         \phi_c(\tau N,0) &\text{if $N$ is non-projective,}\\
         \phi_c(0,N) &\text{if $N$ is projective},
     \end{cases}\]
     \item [(2)] for an indecomposable $\tau$-rigid pair $(0,Q)$, we have \[\tau_c(\phi_c(0,Q))=
         \phi_c(\nu Q,0), 
     \] 
 \end{itemize}
 in particular, for a $\tau$-tilting pair $(M,P)$ of $H$, we have
 \[\tau_c(\tilde{\phi}_c(M,P))=\tilde{\phi}_c(\tau M\oplus \nu P, M_{\mathrm{pr}}).\]

 Dually, 
\begin{itemize}
 \item [(3)] For an indecomposable $\tau$-rigid pair $(N,0)$, we have \[\tau^{-1}_c(\phi_c(N,0))=\begin{cases}
         \phi_c(\tau^{-} N,0) &\text{if $M$ is non-projective,}\\
         \phi_c(0,\nu^{-1}N) &\text{if $N$ is injective},
     \end{cases}\]
     \item [(4)] for an indecomposable $\tau$-rigid pair $(0,Q)$, we have \[\tau_c(\phi_c(0,Q))=
         \phi_c(Q,0), 
     \] 
 \end{itemize}
 in particular, for a $\tau$-tilting pair $(M,P)$ of $H$, we have
 \[\tau^{-1}_c(\tilde{\phi}_c(M,P))=\tilde{\phi}_c(\tau^{-} M\oplus P, \nu^{-1}M_{\mathrm{in}}),\]
 where $M_{\mathrm{in}}$ is the maximal injective summand of $M$.  
\end{theorem}

Let us prepare to show Theorem \ref{description-tauc-rep}. Let $F^{\pm}_k$ be reflection functors, and 
\[C^{+}:=F^+_{i_1}\circ F^+_{i_2}\circ\cdots\circ F^+_{i_n}, \ C^{-}:=F^-_{i_n}\circ F^-_{i_{n-1}}\circ\cdots\circ F^-_{i_1}\] be Coxeter functors on $H$ (where the coressponding Coxeter element is $c=s_{i_1}\cdots s_{i_n}$). These are introduced in \cite{gls-i}. For details of these functors, see \cite{gls-i}*{Sections 9 and 10}. In this paper, we extract some properties we need of them. In the rest of this section, we sometimes identify the rank vector of $M$ with a root corresponding to $M$.  
The following lemma is the special case of \cite{gls-i}*{Proposition 9.4}:
\begin{lemma}\label{rank-reflection}
Let $M\in \mod H$ be an indecomposable locally free module. 
\begin{itemize}
    \item [(1)]
 Let $i$ be a sink vertex in $Q^c$. Then we have the following equation:
\[\underline{\mathrm{rank}}_HF^+_i(M)=\begin{cases}0&\text{if $M\simeq H_i$ as a right $H$-module,}\\
    s_i(\underline{\mathrm{rank}}_HM) & \text{otherwise.}\end{cases}\]
\item [(2)]
Let $i$ be a source vertex in $Q^c$. Then we have the following equation:
\[\underline{\mathrm{rank}}_HF^-_i(M)=\begin{cases}0&\text{if $M\simeq H_i$ as a right $H$-module,}\\
    s_i(\underline{\mathrm{rank}}_HM) & \text{otherwise.}\end{cases}\]
\end{itemize}
\end{lemma}

\begin{lemma}[\cite{gls-i}*{Proposition 9.6}]\label{reflectedlocallyfree}
 For each a $\tau$-rigid module $M\in \mod H$, $F^{\pm}_k(M)$ are $\tau$-rigid again.   
\end{lemma}
\begin{lemma}[\cite{gls-i}*{Theorem 10.1}]\label{TC-tau}
For each locally free module $M\in \mod H$, we have $TC^{+}(M)\simeq \tau(M)$ and $TC^{-}(M)\simeq \tau^-(M)$, where $T$ is the twist automorphism.
\end{lemma}

By the three lemmas mentioned above, we can make the following important observation:

\begin{theorem}\label{coxeter-tau}
For an indecomposable $\tau$-rigid module $M\in \mod H$, we have $C^\pm(M)\simeq\tau^\pm M$ respectively, and $\tau^{\pm} M$ are $\tau$-rigid and locally free. Moreover, if $\tau M\neq 0$ (i.e., $M$ is non-projective), we have
$\underline{\mathrm{rank}}_{H}\tau M=c(\underline{\mathrm{rank}}_{H} M)$. Dually, if $\tau^{-} M\neq 0$ (i.e., $M$ is non-injective), we have $\underline{\mathrm{rank}}_{H}\tau^{-} M=c^{-1}(\underline{\mathrm{rank}}_{H} M)$.
\end{theorem}

\begin{proof}
Since $T$ preserves rigidity and locally freeness, if $M$ is $\tau$-rigid, then $C^{\pm}(M)$ and $TC^{\pm}(M)$ are $\tau$-rigid by Lemma \ref{reflectedlocallyfree} and Theorem \ref{facts-generalized-path-alg} (2). By the uniqueness of the rank vector of $\tau$-rigid modules, we have $C^{\pm}(M)\simeq TC^{\pm}(M)$. Therefore, by Lemma \ref{TC-tau}, we have $C^\pm(M)\simeq\tau^\pm M$. We will now prove the latter part. Since we have $\tau M\simeq C^+(M)=F^+_{i_1}\circ F^+_{i_2}\circ\cdots\circ F^+_{i_n}(M)$, if $\tau M\neq 0$, then $F^+_{i_k}(F^+_{i_{k+1}}\circ\cdots\circ F^+_{i_n}(M))\neq 0$ for each $k$. Therefore, by Lemma \ref{rank-reflection}, we have $\underline{\mathrm{rank}}_{H}\tau M=c(\underline{\mathrm{rank}}_{H} M)$. The dual statement is proven in the same way.
\end{proof}

\begin{proof}[Proof of Theorem \ref{description-tauc-rep}]
 We will prove (1) and (2) only. We will show (1). First, we prove the case that $N$ is projective. We assume that $N=P(i)$, where $P(i)$ is the projective module associated with the vertex $i$ in $Q^c$. It suffices to show that $\tau_c(\underline{\mathrm{rank}}_{H}P(i))=-\alpha_i$. Since $\tau_c$ is a bijection, it is enough to show
\begin{align}\label{proj-tau-}
    \underline{\mathrm{rank}}_{H}P(i)=\tau_c^{-1}(-\alpha_i).
\end{align}

For the left hand side of the above equation, if $i=i_k$, we have
\begin{align*}
    \tau_c^{-1}(-\alpha_i)&=\sigma_{i_n}\circ\dots\circ\sigma_{i_1}(-\alpha_i)\\
    &=\sigma_{i_n}\circ\dots\circ\sigma_{i_k}(-\alpha_i)\\
    &=\sigma_{i_n}\circ\dots\circ\sigma_{i_{k+1}}(\alpha_i)\\
    &=s_{i_n}\circ\dots\circ s_{i_{k+1}}(\alpha_i).
\end{align*}
This root is given as follows: for a vertex $j\in Q^c$,
\begin{itemize}
    \item if there is not a path from $i$ to $j$ in $Q^c$, then the coefficient $b_j$ of $\alpha_j$ in the simple root decomposition of $\tau_c^{-1}(-\alpha_i)$ is $0$,
    \item the coefficient  $b_i$ of $\alpha_i$ is $1$,
    \item if there is a (unique) path $i=i_0\to i_1\to \dots\to i_{m-1}\to i_{m}=j$ from $i$ to $j$ in $Q^c$, then the coefficient $b_j$ of $\alpha_j$ is given by $|C_{ji_{m-1}}|b_{i_{m-1}}$ inductively. 
\end{itemize}
On the other hand, $\underline{\mathrm{rank}}_{H}P(i)$ is given by using the charactorization of projective modules in representation theory of finite dimensional $K$-algebra as follows (cf. \cite{ass}*{Chapter III}):
\begin{itemize}
    \item if there is not a path from $i$ to $j$ in $Q^c$, then we have $P(i)_j=0$, and thus $\text{rank}_{H_j}P(i)_j=0$,
    \item the $K$-basis of $P(i)_i$ is $\{e_i,\varepsilon_i,\varepsilon_i^2,\dots,\varepsilon_i^{d_i-1}\}$. Therefore, $\dim_K P(i)_i=d_i$ and thus $\text{rank}_{H_i}P(i)_i=1$,
    \item if there is a (unique) path $i=i_0\to i_1\to \dots\to i_{m-1}\to i_{m}=j$ from $i$ to $j$ in $Q^c$, then the $K$-basis of $P(i)_j$ is
    \begin{align*}
    \{\varepsilon_i^{a_0}\alpha_{ii_1}\varepsilon_{i_1}^{a_{1}}\alpha_{i_1i_2}\varepsilon_{i_2}^{a_{2}}\dots \varepsilon_{i_{m-1}}^{a_{m-1}}&\alpha_{i_{m-1}j}\varepsilon_j^b \mid \\ &0\leq a_k \leq |C_{i_{k+1}i_k}|-1\text{ for each $k$, } 0\leq b \leq d_j-1\}.
    \end{align*}
    Therefore, the rank of $(P_i)_j$ as $H_j$-module is given by $|C_{ji_{m-1}}|\cdot\text{rank}_{H_{i_{m-1}}}{P(i)_{i_{m-1}}}$ inductively. 
\end{itemize}
By comparing these descriptions, we have \eqref{proj-tau-}. Next, we prove the case that $M$ is non-projective. By the above discussion, a root $\alpha$ maps to one of negative simple roots by $\tau_c$ if and only if $\phi_c^{-1}(\alpha)$ is a projective module. Therefore, if $N$ is non-projective, then we have $\tau_c(\phi_c(N,0))=c(\phi_c(N,0))$ and the desired equation obtain from Theorem \ref{coxeter-tau}.
We will show (2). By duality of \eqref{proj-tau-}, we have
\[\underline{\text{rank}}_HI(i)=\tau_c(-\alpha_i),\]
where $I(i)$ is the injective module associated with a vertex $i$. Since $I(i)=\nu P(i)$, it finishes the proof of (2).
\end{proof}

\begin{remark}
 In the situation that $H$ is isomorphic to a path algebra $KQ$ associated with a simply-laced Dynkin quiver $Q$, Theorem \ref{description-tauc-rep} implies that $\tau_c$ can be regarded as the AR-translation $\tau$ of the cluster category $\mathcal{C}(Q):=D^{b}(KQ)/\tau^{-1}[1]$ via the Adachi-Iyama-Reiten bijection \cite{air}*{Theorem 4.1}. 
\end{remark}

\subsection{Direct proof of Theorem \ref{lattice-iso-phi}}In this subsection, we prove Theorem \ref{lattice-iso-phi}. By abusing the notation for an indecomposable $\tau$-tilting pair $(N,Q)$ in $\stautilt H$, we denote
\[\tau_c(N,Q):=\phi_c^{-1}\tau_c(\phi_c(N,Q)),\quad R_c(N,Q):=R_c(\phi_c(N,Q)).\]

By Theorem \ref{description-tauc-rep}, we have
\begin{align*}
    \tau_c(N,0)&=\begin{cases}
         (\tau N,0) &\text{if $N$ is non-projective,}\\
         (0,N) &\text{if $N$ is projective},
     \end{cases}\\
     \tau_c(0,Q)&=(\nu Q,0), \\
     \tau^{-1}_c(N,0)&=\begin{cases}
         (\tau^- N,0) &\text{if $N$ is non-injective,}\\
         (0,\nu^{-1}N) &\text{if $N$ is injective},
     \end{cases}\\
     \tau^{-1}_c(0,Q)&=(Q,0).
         \end{align*}

According to Theorem \ref{lattice-iso-phi}, it suffices to show that
\begin{align}\label{exchange-R2}
(M,P) \to (M',P')\ \text{in}\ \overrightarrow{\Gamma}(\stautilt H)\Leftrightarrow R(N,Q)>R(N',Q'),
\end{align}
where $(N,Q)$ (resp., $(N',Q')$) is an indecomposable summand of $(M,P)$ (resp., $(M',P')$) that is not included in $(M',P')$ (resp., $(M,P)$).
Furthermore, \eqref{exchange-R2} can be rephrased by the following lemma:

\begin{lemma}\label{green-preserve-lemma2}
The following condition is equivalent to \eqref{exchange-R2}:
\begin{align}
&\text{For $(M,P)\to(M',P)$ in $\overrightarrow{\Gamma}(\stautilt H)$, }\text{there exists an arrow $\tau^{-1}_c(M,P)\to \tau^{-1}_c(M',P)$.}\label{green-preservation2} 
\end{align}
\end{lemma}

We remark that in Lemma \ref{green-preserve-lemma2}, the two $\tau$-tilting pairs $(M,P)$ and $(M',P)$ have the same projective support $P$ (that is, $(M',P)$ is not a typo for $(M',P')$).

To facilitate understanding of the proof, we shall prove the following lemma.

\begin{lemma}\label{change-components}
We assume that $(M,P)$ is connected with $(M',P')$
in ${\Gamma}(\stautilt H)$, and let $(N,Q)$ (resp., $(N',Q')$) be an indecomposable summand of $(M,P)$ (resp., $(M',P')$) that is not included in $(M',P')$ (resp., $(M,P)$). Then, $Q'\neq 0$ holds if and only if there is an arrow $(M,P) \to (M',P')$
in $\overrightarrow{\Gamma}(\stautilt H)$ and $P'\neq P$ holds. In this situation, we have $Q=0$.
\end{lemma}

\begin{proof}
First, we prove (1). We assume that $Q'\neq 0$. If $Q$ was nonzero, then we have $M=M'$. Since $P$ is determined by $M$, we have $P=P'$. It is a conflict with the assumption that $Q'$ is not a summand of $P$. Therefore, we have $Q=0$ and $P'\neq P$. Moreover, since $(N,Q)$ is indecomposable and $Q=0$, we have $N\neq 0$. Therefore, we have $M\simeq M'\oplus N$ and $(M,P)\to(M',P')$ in $\overrightarrow{\Gamma}(\stautilt H)$. We show the converse. We assume that there is an arrow $(M,P) \to (M',P')$
in $\overrightarrow{\Gamma}(\stautilt H)$ and $P'\neq P$. For $P$ and $P'$, there are three possible cases: (i) $P\subset P'$, (ii) $P\supset P'$, (iii) there is no inclusion relation between $P$ and $P'$. In fact, (ii) and (iii) cannot happen. Indeed, if (ii) held, then we have $Q\neq0$ and $N'\neq 0$, and we have $M'\simeq M\oplus N'$. It conflicts with $(M,P) \to (M',P')$
in $\overrightarrow{\Gamma}(\stautilt H)$. If (iii) held, we have $Q\neq 0$ and $Q'\neq 0$. Then, we have $M=M'$ and $P=P'$. It conflicts with $P\neq P'$. Therefore, we have (i) and $Q'\neq 0$ and $Q=0$.
\end{proof}

\begin{proof}[Proof of Lemma \ref{green-preserve-lemma2}]
We assume that \eqref{exchange-R2} holds for any $(M,P),(M',P')$. We consider $(M,P)\to (M',P)$ in $\overrightarrow{\Gamma}(\stautilt H)$. By Lemma \ref{change-components}, we have $Q=Q'=0$, and we have $R(N,Q)\neq 0$ and $R(N',Q')\neq 0$. Thus $R(\tau_c^{-1}(N,Q))=R(N,Q)-1$ and $R(\tau_c^{-1}(N',Q'))=R(N',Q')-1$ hold, and we have $R(\tau_c^{-1}(N,Q))>R(\tau_c^{-1}(N',Q'))$. Therefore we have an arrow $\tau^{-1}_c(M,P)\to \tau^{-1}_c(M',P)$ in $\overrightarrow{\Gamma}(\stautilt H)$. Next, we assume that \eqref{exchange-R2} does not hold for some $(M,P), (M',P')$. Choose $(M,P), (M',P')$ such that $(M,P)\to (M',P')$ and $R(N,Q)<R(N',Q')$. If $Q'\neq 0$ held, then we have $(M,P)\to (M',P')$ and $R(N,Q)>R(N',Q')=0$ by Lemma \ref{change-components}, and it is a contradiction. Therefore $Q'=0$ must hold. Moreover, we have $Q= 0$ and $P'=P$ by Lemma \ref{change-components}. Then, if the green-preservation condition \eqref{green-preservation2} was true, since the second components of $(N',Q'),\tau^{-1}_c(N',Q'),\dots,\tau_c^{-R(N,Q)+1}(N',Q')$ are $0$, the second components of $\tau^{-i}_c(M,P)$ and $\tau^{-i}_c(M',P)$ coincide for $0\leq i\leq R(N,Q)-1$ by Lemma \ref{change-components}, and we have an arrow $\tau^{-R(M,P)}_c(M,P)\to\tau^{-R(M,P)}_c(M',P)$. However, the second component of $\tau^{-R(M,P)}_c(M,P)$ is non-zero, and it conflicts with Lemma \ref{change-components}.
\end{proof}

The following lemma is used in the proof of Theorem \ref{lattice-iso-phi}:

\begin{lemma}[\cite{air}*{Lemma 2.25}]\label{equiv-arrow}
Let $A$ be a $K$-algebra. The following are equivalent:
\begin{itemize}
    \item [(i)] $(M,P)\to(M',P')$ in $\overrightarrow{\Gamma}(\stautilt A)$
    \item[(ii)] $\Hom_H(M',\tau M)=0$ and $\add P\subset \add P'$
\end{itemize}
\end{lemma}

We are ready to prove Theorem \ref{lattice-iso-phi}.

\begin{proof}[Proof of Theorem \ref{lattice-iso-phi}]
By Lemma \ref{green-preserve-lemma2}, it suffices to show \eqref{green-preservation2}. By Theorem \ref{description-tauc-rep}, we have $\tau_c^{-1}(M,P)=(\tau^-M\oplus P,\nu^{-1} M_{\text{in}})$.  By Lemma \ref{equiv-arrow}, it is enough to show that if $\Hom_H (M',\tau M)=0$, then we have $\Hom_H (\tau^{-} M',\tau\tau^{-} M)=0$ and $\add \nu^{-1}M_{\mathrm{in}}\subset \add \nu^{-1}M'_{\mathrm{in}}$. We will show that $\Hom_H (\tau^{-} M',\tau\tau^{-} M)=0$. Since $\mathrm{proj. dim~} M\leq 1$ and $\mathrm{proj. dim~} M'\leq 1$ by Theorem \ref{facts-generalized-path-alg} (1) and (2), we have 
\begin{align*}
    0=\Hom_H(M',\tau M)\simeq \Hom_H(\tau^{-}M',M). 
\end{align*}
Since $\tau\tau^-M$ is a summand of $M$, we have $\Hom_H (\tau^{-} M',\tau\tau^{-} M)=0$. Next, we prove $\add \nu^{-1}M_{\mathrm{in}}\subset \add \nu^{-1}M'_{\mathrm{in}}$. It suffices to show $\add M_{\mathrm{in}}\subset \add M'_{\mathrm{in}}$. Let $N$ be a (unique) indecomposable summand of $M$ which is not a summand of $M'$. It suffices to show that $N$ is not injective. We assume that $N$ is injective. We set $M_0$ as a module satisfying $M\simeq M_0\oplus N$. Using the same argument as in the proof of Theorem \ref{condition-tilting-true}, we can regard $M$ and $M'$ as tilting $\Lambda/\Lambda e\Lambda$-modules (where $e$ is the idempotent satisfying $P\simeq e\Lambda$). Therefore, we have an exact sequence
\[0\rightarrow N \xrightarrow{f} \tilde{M_0} \xrightarrow{g} N'\rightarrow 0,\]
where $\tilde{M_0}\in \add M_0$ and $M'\simeq M_0\oplus N'$. Since $N$ is injective, this sequence splits. Therefore, we have $\tilde{M_0}\simeq N\oplus N'$. This implies that $N$ is a summand of $M_0$, and it is contradiction. Therefore, we have $\add M_{\mathrm{in}}\subset \add M'_{\mathrm{in}}$.
\end{proof}
\section{Isomorphism between Cambrian and torsion lattice}

In this section, we construct the lattice isomorphism between the torsion lattice of $H$ and the $c$-Cambrian lattice after introducing the torsion lattice and $c$-Cambrian lattice. Refer \cite{ass} for more information on the torsion class, and \cite{dirrt} for a detailed explanation of how the Cambrian lattice and representation theory are related.

\subsection{Torsion lattice}

This section describes the torsion lattice. A \emph{torsion pair} $(\mathcal{T},\mathcal{F})$ in $\mod A$ is a pair of subcategories satisfying the following conditions:
\begin{itemize}
\item[(i)] $\Hom_A(T,F)=0$ for any $T\in \mathcal{T}$ and $F\in \mathcal{F}$;
\item[(ii)] For any module $X\in \mod A$, there exists a short exact sequence
 $$0\rightarrow X^{\mathcal{T}}\rightarrow X\rightarrow X^{\mathcal{F}}\rightarrow0$$ with $X^{\mathcal{T}}\in\mathcal{T}$ and $X^{\mathcal{F}}\in \mathcal{F}$. Due to the criterion (i), such a sequence is unique up to isomorphisms. This short exact sequence is called the \emph{canonical sequence} of $X$ with respect to $(\mathcal{T},\mathcal{F})$.
\end{itemize}

The subcategory $\mathcal T$ (resp., $\mathcal F$) in a torsion pair $(\mathcal T,\mathcal F)$ is called a \emph{torsion class} (resp., \emph{torsionfree class}) of $\mod A$. A torsion class $\mathcal T$ is said to be \emph{functorially finite} if it is functorially finite as a subcategory of $\mod A$, that is, $\mathcal T$ is both contravariantly and covariantly finite in $\mod A$, cf. \cite{AS_1981}.

\begin{definition}[Torsion lattice]
Let $\mathsf{tors} A$ be the set of torsion classes of $A$. The set $\mathsf{tors} A$ forms a lattice structure by inclusion. This lattice is called the \emph{torsion lattice}.
\end{definition}

\begin{example}[Type $A_2$]
We set $A=K(1\leftarrow 2)$. Then the torsion lattice of $A$ is as follows:
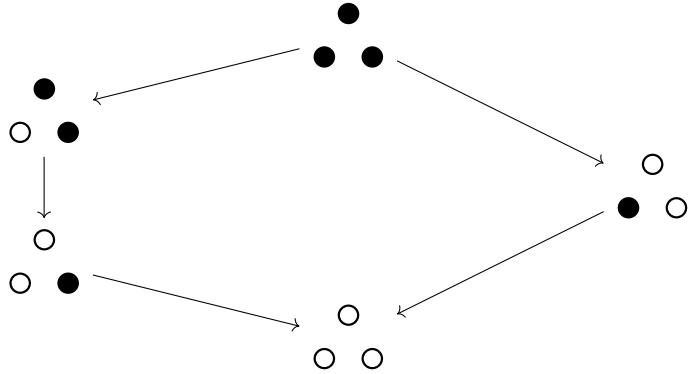
\[\begin{tikzpicture}
      \node (1) at (0,6) {\Huge $\sst{\bullet\\ \bullet\ \bullet}$};
      \node(2) at (-4,5) {\Huge $\sst{\bullet\\ \circ\ \bullet}$};
      \node (3) at (4,4) {\Huge $\sst{\circ\\ \bullet\ \circ}$};
      \node (4) at (-4,3) {\Huge $\sst{\circ\\ \circ\ \bullet}$};
      \node (5) at (0,2) {\Huge $\sst{\circ\\ \circ\ \circ}$};
      \draw[->] (1)--(2);
      \draw[->] (1)--(3);
      \draw[->] (2)--(4);
      \draw[->] (3)--(5);
      \draw[->] (4)--(5);
    \end{tikzpicture}\]
where the three cycles in each vertex represents the AR-quiver of $A$,
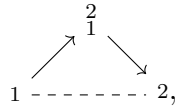
\begin{center}
$\begin{tikzpicture}
 \node (0) at (0,0) {${\substack{1}}$};
 \node (1) at (1,1){$\substack{2 \\ 1}$};
 \node (2) at (2,0){$\substack{2},$};
 \draw[->](0)--(1);
 \draw[->](1)--(2);
 \draw[dashed](0)--(2);
\end{tikzpicture}$
\end{center}
and black circles indicate that the corresponding module belongs to the torsion class.
\end{example}

Next, we recall the relation between torsion classes and $\tau$-tilting pairs.

\begin{theorem}[\cite{air}*{{Proposition 1.2 (b) and Theorem 2.7}}]\label{proorder}
There exists a well-defined map $\Psi_A$ from $\tau$-rigid pairs to functorially finite torsion classes in $\mod A$,
given by $(M,P)\mapsto \Fac M$.
Moreover, $\Psi_A$ is a bijection if we restrict it to basic $\tau$-tilting pairs.
\end{theorem}

Generally, the set of functorially finite torsion classes of $A$ does not coincide with the set of torsion classes of $A$. However, the following theorem demonstrates that they do coincide in special circumstances.

\begin{theorem}[\cite{dij}*{Theorem 3.1}]\label{tors=f-tors}
The algebra $A$ is $\tau$-tilting finite if and only if the set of functorially finite torsion classes coincides with the set of torsion classes.
\end{theorem}

By Theorems \ref{proorder} and \ref{tors=f-tors}, we have the following theorem:

\begin{theorem}\label{supp-tau-torsion}
If $A$ is $\tau$-tilting finite, then $\Psi_A$ is a quiver isomorphism from $\overrightarrow{\Gamma}(\stautilt A)$ to $\Hasse(\mathsf{tors} A)$.
\end{theorem}

\subsection{Cambrian lattice}
We define the Cambrian lattice. We fix a root system $\Phi$ of the Dynkin-type. Let $C_\Phi$ be the Cartan matrix and $W(\Phi)$ the Weyl group corresponding to $\Phi$. We denote by $\{s_1,\dots,s_n\}$ the simple reflections of $W(\Phi)$. Let $c=c_1\cdots c_n \in W(\Phi)$ be a Coxeter element, where $\{c_1,\dots,c_n\}=\{s_1,\dots,s_n\}$. We denote by
\[c^{\infty}=c_1\cdots c_nc_1 \cdots c_n c_1 \cdots\]
the half-infinity words in which $c$ appears repeatedly.

To define the $c$-sorting word, we define a total order $\leq_c$ associated with a Coxeter element $c$ over the set of finite words consisting of the letters $s_1,\dots,s_n$. For any two words $w_1=s_{{w_1}1}\cdots s_{{w_1}m}$ and $w_2=s_{{w_2}1}\cdots s_{{w_2}m'}$ with $w_1\neq w_2$, we define the size relation of $w_1$ and $w_2$ as follows:
\begin{itemize}
\item [(1)] If $w_2$ (resp. $w_1$) is represented by $w_2=w_1s_{{w_2}i}\cdots s_{{w_2}m'}$ (resp. $w_1=w_2s_{{w_1}i}\cdots s_{{w_1}m}$), then we set $w_1<_cw_2$ (resp. $w_2<_cw_1$). 
\item[(2)] We assume that $w_1$ and $w_2$ does not satisfy the case (1). Then, there exists a natural number $k$ such that $s_{w_1i}=s_{w_2i}$ for any $i<k$ and $s_{w_1k}\neq s_{w_2k}$. We find $s_{w_11},s_{w_12},\dots,s_{w_1k-1}$ in order from the left in $c^\infty$, and in the string after this in $c^{\infty}$, if $s_{w_1k}$ (resp. $s_{w_2k}$) appears before $s_{w_2k}$ (resp. $s_{w_1k}$), then we set $w_1<_cw_2$ (resp. $w_2<_cw_1$).
\end{itemize}
\begin{example}
    We set a Coxeter element $c=s_1s_2s_3s_4$, 
and two words $w_1=\textcolor{violet}{s_1s_3s_2}\textcolor{blue}{s_1}$ and $w_2=\textcolor{violet}{s_1s_3s_2}\textcolor{red}{s_4}$. Since $c^{\infty}=\textcolor{violet}{s_1}s_2\textcolor{violet}{s_3}s_4s_1\textcolor{violet}{s_2}s_3\textcolor{red}{s_4}\textcolor{blue}{s_1}s_2s_3s_4\cdots$, we have $w_2<_cw_1$. 
\end{example}
\begin{definition}[$c$-sorting word]
We fix a Coxeter element $c=c_1\cdots c_n \in W(\Phi)$. The \emph{$c$-sorting word} for $w \in W(\Phi)$ is a minimal reduced form of $w$ with respect to the order $\leq_c$.
\end{definition}

By inserting a divider for each word representing $c$ in $c^{\infty}$,
\[c_1\cdots c_n|c_1 \cdots c_n|c_1 \cdots,\]
a word $w$ is considered as a sequence of a subset of $\{s_1,\dots,s_n\}$, i.e., sets of letters that appear between adjacent dividers.

\begin{example}
We set a Coxeter element $c=s_1s_2s_3s_4$ and a word $w=s_1s_3s_2s_4s_1$. Then, since $c^{\infty}=|\textcolor{red}{s_1}s_2\textcolor{red}{s_3}s_4|s_1\textcolor{red}{s_2}s_3\textcolor{red}{s_4}|\textcolor{red}{s_1}s_2s_3s_4|\cdots$, the sequence of subsets of $\{s_1,s_2,s_3,s_4\}$ corresponding to $w$ is
\[(\{s_1,s_3\},\{s_2,s_4\},\{s_1\}).\]
\end{example}

\begin{definition}[$c$-sortable]
An element $w \in W(\Phi)$ is \emph{$c$-sortable} if its $c$-sorting word defines a gradually diminishing sequence of subsets under inclusion.
\end{definition}
We note that the set of $c$-sortable elements is independent of the word choice used to truncate $c$.

\begin{example}[Type $A_2$]\label{ex-c-sortable}
Let $\Phi$ be of $A_2$ type, and $c=s_2s_1$. Then, the $c$-sortable elements of $W(\Phi)$ are $1, s_1,s_2,s_2s_1,s_2s_1s_2$. An element $s_1s_2$ is not $c$-sortable.
\end{example}

\begin{definition}[Inversion set]
For any $w\in W(\Phi)$, we denote by
\[\mathrm{I}(w)=\{\alpha\in \Phi^+ \mid w^{-1}(\alpha) \in -\Delta\},\]
where $-\Delta$ is the set of negative roots in $\Phi$ and $\Phi^+$ is the set of positive roots. The set $\mathrm{I}(w)$ is called the \emph{inversion set of $w$}.
\end{definition}

\begin{definition}(Cambrian lattice)
We fix a Dynkin-type root system $\Phi$ and a Coxeter element $c\in W(\Phi)$. The set of $c$-sortable elements in $W(\Phi)$ exhibits a lattice structure when ordered by inclusion of inversion sets. This lattice is called a \emph{$c$-Cambrian lattice} and is denoted by $\mathfrak C_c$.
\end{definition}

\begin{example}[Type $A_2$]
Consider the same situation as in Example \ref{ex-c-sortable}. The inversion set for each $c$-sortable element is as follows:
\begin{align*}
    \mathrm{I}(1)=\emptyset,\ \mathrm{I}(s_1)=\{\alpha_1\},\ \mathrm{I}(s_2)=\{\alpha_2\},\
    \mathrm{I}(s_2s_1)=\{\alpha_2,\alpha_1+\alpha_2\},
     \mathrm{I}(s_2s_1s_2)=\{\alpha_1,\alpha_2,\alpha_1+\alpha_2\}.
\end{align*}
Therefore, the Hasse quiver of $c$-Cambrian lattice is as follows:
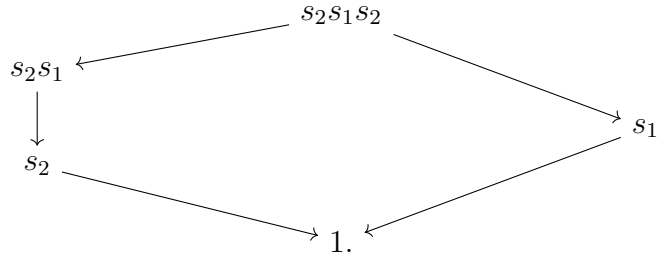
\[\begin{tikzpicture}
      \node (1) at (0,5) {$s_2s_1s_2$};
      \node(2) at (-4,4.25) {$s_2s_1$};
      \node (3) at (4,3.5) {$s_1$};
      \node (4) at (-4,3) {$s_2$};
      \node (5) at (0,2) {$1.$};
      \draw[->] (1)--(2);
      \draw[->] (1)--(3);
      \draw[->] (2)--(4);
      \draw[->] (3)--(5);
      \draw[->] (4)--(5);
    \end{tikzpicture}\]
\end{example}
We examine the relation between the Hasse quiver of $c$-Cambrian lattice and the quiver of $c$-clusters.

\begin{definition}[Positive root associated with last reflection]
Let $a=a_1a_2\cdots a_k$ be the $c$-sorting word for $w$. If $s_i$ occurs in $a$, then the \emph{positive root associated with the last reflection of $s_i$} in $w$ is $a_1a_2\cdots a_{j-1}(\alpha_i)$, where $a_j$ is the rightmost occurrence of $s_i$ in $a$.
\end{definition}

The set $\textrm{cl}_c(w)$ is comprised of the set of all positive roots associated with the last reflection of $w$ and the set of negative simple roots for each $s_i$ that do not appear in $a$.

\begin{theorem}[\cite{re-sp}*{Corollary 8.1 and Proposition 8.3}]\label{cambrian-cluster}
The map $\mathrm{cl}_c\colon \Hasse(\mathfrak C)\to \overrightarrow{\Gamma}(\apPhi,c)$ is a quiver isomorphism.
\end{theorem}

According to Theorem \ref{cambrian-cluster}, the set of $c$-clusters has a lattice structure whose Hasse quiver is the quiver of $c$-clusters; this structure is called the \emph{$c$-cluster lattice} (\cite{re-sp}*{Corollary 8.5}).

\subsection{Proof of Theorem \ref{main2-intro}}

It is sufficient to create a quiver isomorphism between their Hasse quivers. This facts is supported by the following lemma:

\begin{lemma}\label{tors-silt-iso}\noindent
    The map $\Psi_H\colon \overrightarrow{\Gamma}(\stautilt H)\to \Hasse(\mathsf{tors} H)$ is a quiver isomorphism.
\end{lemma}
\begin{proof}
The statement follows from Theorems \ref{supp-tau-torsion} and \ref{psi+}.
\end{proof}

Thus, we have the following corollary:

\begin{corollary}\label{main}
The map $\Psi_H\circ \tilde\phi_c^{-1}\circ \mathrm{cl}_c\colon \Hasse(\mathfrak C_c)\to \Hasse(\mathsf{tors}H)$ is a quiver isomorphism. Particularly, the lattice $\mathfrak C_c$ is isomorphic to $\mathsf{tors}H$.
\end{corollary}

\begin{proof}
It follows from Theorems \ref{cambrian-cluster}, \ref{lattice-iso-phi} and Lemma \ref{tors-silt-iso}. See the below diagram:

\[\begin{xy}
  (30,15)*{\overrightarrow{\Gamma}(\stautilt H)}="B",(80,15)*{\Hasse(\mathsf{tors}H)}="C",(30,-15)*{\overrightarrow{\Gamma}(\apPhi,c)}="D",(80,-15)*{\Hasse(\mathfrak C_c)}="E",\ar^{\Psi_H}@{->} "B";"C"\ar_{\mathrm{cl}_c}@{<-}"D";"E",\ar^{\tilde\phi_c}@{->}"B";"D",\ar@{<.}"C";"E"
\end{xy}\]
\end{proof}

\begin{remark}
When $\Phi$ is simply-laced and $D_\Phi=I_n$, $\Psi_H\circ \tilde\phi_c^{-1}\circ \mathrm{cl}_c$ is given by $w\mapsto \add((\phi^+_c)^{-1}(\mathrm{I}(w)))$ \cite{ing-tho}. This is not generally true. Indeed, for any $c$-sortable element $w$, we have \[\add((\phi^+_c)^{-1}(\mathrm I(w)))\subset\add(\tau\textsf{-rigid}H).\] On the other hand, $\mod H$ is a torsion class, therefore there exists $c$-sortable element $w$ such that $\Psi_H\circ \tilde\phi_c^{-1}\circ \mathrm{cl}_c(w)=\mod H$. However, except for the above situation, we have $\add(\tau\textsf{-rigid}H)\subsetneq \mod H$ because $H$ is $\tau$-tilting-finite but not representation-finite. Therefore, we have $\add((\phi^+_c)^{-1}(\mathrm I(w)))\neq \Psi_H\circ \tilde\phi_c^{-1}\circ \mathrm{cl}_c(w)$.
\end{remark}

\begin{question}
How is the direct description of the bijective correspondence from $\mathfrak C_c$ to $\mathsf{tors}H$ given?
\end{question}

\begin{remark}\label{important-remark}
Corollary \ref{main} can also be proved using some previous studies, which do not rely on Theorem \ref{lattice-iso-phi}, as follows. The following facts are known:
\begin{itemize}
    \item[(1)] (\cite{eno22}*{Theorem C}) In an abelian length category with a finite number of torsion classes, the lattice given by the poset of all wide subcategories is isomorphic to the lattice given by the set of all torsion classes ordered by the shard intersection order.
    \item[(2)] (\cite{Gratzer}*{Theorem 10-6.34}) The lattice given by the poset of all non-crossing partitions is isomorphic to the lattice given by the set of all $c$-sortable elements ordered by the shard intersection order.
\end{itemize}
By these facts, to prove Corollary \ref{main}, it suffices to show the lattice of all wide subcategories isomorphic to the lattice of all non-crossing partitions. Since the isomorphism between the lattice of all wide subcategories of a finite dimensional hereditary algebra and the corresponding lattice of all non-crossing partitions is given by \cite{hub-kra}*{Theorem 1.2}, this claim can be proved via the relation between GLS path algebras and finite dimensional hereditary algebras given by \cite{gls-reg}*{Theorem 1.1} and Theorem \ref{supp-tau-torsion}. Moreover, by using this fact, we can give another proof of Theorem \ref{lattice-iso-phi}.
\end{remark}

\bibliography{myrefs}
\end{document}